\documentclass[12pt,reqno]{amsart}

\usepackage[hmargin=3cm,vmargin=3cm]{geometry}
\usepackage{amsmath,amssymb}
\usepackage{graphicx}
\usepackage{mathtools}
\usepackage{thmtools,thm-restate}
\usepackage{color}

\usepackage{hyperref}

\usepackage{enumitem}

\newtheorem{theorem}{Theorem}[section]
\newtheorem{proposition}[theorem]{Proposition}
\newtheorem{corollary}[theorem]{Corollary}

\newtheorem{lemma}[theorem]{Lemma}

\theoremstyle{definition}
\newtheorem{definition}[theorem]{Definition}

\newtheorem*{definition*}{Definition}

\theoremstyle{remark}
\newtheorem{remark}[theorem]{Remark}

\numberwithin{equation}{section}

\def\NN{\mathbb{N}}
\def\RR{\mathbb{R}}

\def\ZZ{\mathbb{Z}}

\def\Sone{\mathbb{S}^1}

\newcommand{\cB}{{\mathcal B}}

\newcommand{\cE}{{\mathcal E}}

\newcommand{\cH}{{\mathcal H}}
\newcommand{\cI}{{\mathcal I}}
\newcommand{\cJ}{{\mathcal J}}

\newcommand{\cL}{{\mathcal L}}

\newcommand{\cP}{{\mathcal P}}

\newcommand{\cT}{{\mathcal T}}

\newcommand{\cZ}{{\mathcal Z}}
\newcommand{\sign}{\operatorname{sign}}

\newcommand{\tr}{\operatorname{tr}}
\newcommand{\myspan}{\operatorname{span}}

\newcommand{\rank}{\operatorname{rank}}

\DeclareMathOperator\Div{div}

\DeclareMathOperator\curl{curl}
\DeclareMathOperator\Imag{Im}
\DeclareMathOperator\Bd{bdry}
\DeclareMathOperator\Int{int}
\newcommand{\id}{{\rm id}}
\DeclareMathOperator\RegImag{RegIm}
\DeclareMathOperator\diam{diam}
\DeclareMathOperator\dist{dist}
\DeclareMathOperator\hess{Hess}

\author{David Perrella}
\address{David Perrella,
The University of Western Australia, 35 Stirling Highway, Crawley WA 6009, Australia \newline \it{e-mail: david.perrella@uwa.edu.au}}

\thanks{The author thanks Nathan Duignan and David Pfefferl\'e for feedback on the paper. The author also thanks Nathan for providing the curve in Figure \ref{fig:solid-torus-example}, along with important ideas for its construction in Appendix \ref{app:for-the-figure}. The author would lastly like to thank the mathematics department of the University of Western Australia for providing a space to continue research after completion of his PhD}

\title{Closed orbits of MHD equilibria with orientation-reversing symmetry}

\begin{document}

\begin{abstract}
As a generalisation of the periodic orbit structure often seen in reflection or mirror symmetric MHD equilibria, we consider equilibria with other orientation-reversing symmetries. An example of such a symmetry, which is a not a reflection, is the parity transformation $(x,y,z) \mapsto (-x,-y,-z)$ in $\RR^3$. It is shown under any orientation-reversing isometry, that if the pressure function is assumed to have toroidally nested level sets, then all orbits on the tori are necessarily periodic. The techniques involved are almost entirely topological in nature and give rise to a handy index describing how a diffeomorphism of $\RR^3$ alters the poloidal and toroidal curves of an invariant embedded 2-torus.
\end{abstract}

\maketitle

\section{Introduction}

This paper is primarily concerned with the periodic nature of MHD equilibria with symmetry by orientation-reversing isometry. A by-product of the analysis gives a topological characterisation of how diffeomorphisms of $\RR^3$ act on (the first homology of) an invariant 2-torus, which could be useful in other problems studying symmetries of magnetic fields. By MHD equilibrium, it is meant a solution $(B,p)$ to the set of equations
\begin{equation*}
\curl B \times B = \nabla p, \qquad \Div B = 0, \qquad B \cdot n = 0,
\end{equation*}
where $B$ is a smooth vector field, $p$ is a smooth function, and $n : \partial M \to TM$ denotes the outward unit normal on an orientable Riemannian 3-manifold $M$ with boundary. 

The case of most interest in magnetic confinement fusion is when $M \subset \RR^3$ is a solid toroidal domain and $p$ has toroidally nested level sets. In this case, the boundary condition $B \cdot n = 0$ is implied by $B \cdot \nabla p = 0$. By $M$ being a solid toroidal domain, we mean that $M$ is embedded in $\RR^3$ and is diffeomorphic to $\mathbb{S}^1 \times D$, where $\Sone = \mathbb{R}/\mathbb{Z}$ is the circle and $D$ is the closed unit disk in $\RR^2$. The meaning of toroidally nested level sets is best conveyed in a picture (see Figure \ref{fig:solid-torus-example}). See Definition \ref{def:toroial-lvl-sets} in Section \ref{sec:extension-to-nested-toroidal-surfaces} for the precise meaning used in this paper.

\begin{figure}\label{fig:solid-torus-example}
    \centering
    \includegraphics[width=0.8\linewidth]{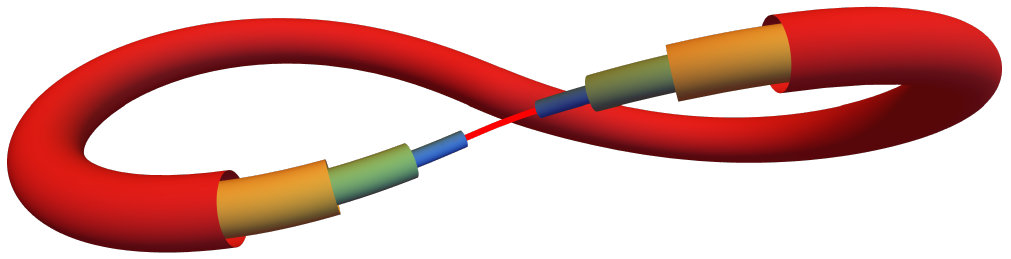}
    \caption{The level sets of $p(x) = \dist(x,\Gamma)^2$, the distance squared to the image $\Gamma$ of the (red) curve $\gamma(t) = (r(t) \cos t \sin \theta(t),r(t) \sin t  \sin\theta(t),r(t)\cos\theta(t))$ with $r(t) = 2+\cos 2t$ and $\theta(t) = \pi/2+\sin(2t)/5$. The curve $\gamma$ and its depicted tubular neighbourhood $\{\dist(x,\Gamma) \leq R\}$ for $R > 0$ sufficiently small have only $(x,y,z) \mapsto (-x,-y,-z)$ as a non-trivial Euclidean symmetry. The function $p$ has toroidally nested level sets with axis $\Gamma$.}
\end{figure}

Geometric symmetries of a domain $M \subset \RR^3$ are often exploited to construct MHD equilibria. In fact, there is the outstanding conjecture of Grad \cite{Grad67,Grad85} which states loosely that if $M \subset \RR$ is a solid toroidal domain and $(B,p)$ is an MHD equilibrium on $M$ where $p$ has toroidally nested level sets, then $M$, $B$, and $p$ admit a mutual Euclidean symmetry. See also \cite{Constantin21a} for more precision on the conjecture.

The most common symmetry used to construct MHD equilibria is continuous rotational symmetry about some axis, known as axisymmetry. Another symmetry which is less common is that of reflection or mirror symmetry. Lortz \cite{Lortz70} constructed MHD equilibria in solid toridal domains which possess reflection symmetry. More precisely, up to rotation and translation, Lortz considers toroidal domains $M \subset \RR^3$ which are preserved by the transformation $(x,y,z) \mapsto (-x,y,z)$ in $\RR^3$; if $\Psi : M \to M$ denotes the corresponding diffeomorphism, then the MHD equilibria $(B,p)$ on $M$ constructed by Lortz satisfy
\begin{equation*}
\Psi_* B = -B, \qquad \Psi^*p = p,    
\end{equation*}
where $\Psi_*$ denotes the pushfoward and $\Psi^*$ denotes the pullback by $\Psi$. Moreover, the constructed $B$ have no zeros, all orbits of $B$ are periodic, and the plane of reflection $\{x=0\}$ intersects $M$ in two disjoint closed disks. 

Complementary to the set-up of Lortz, Salat \cite{Salat95} considers (up to rotation and translation) toroidal domains $M \subset \RR^3$ which are preserved by the transformation $(x,y,z) \mapsto (x,y,-z)$ in $\RR^3$. In this set-up, the plane of reflection $\{z = 0\}$ intersects $M$ in a single annulus. If $\Psi : M \to M$ denotes the corresponding diffeomorphism, then the MHD equilibria $(B,p)$ on $M$ are assumed to satisfy
\begin{equation*}
\Psi_* B = -B, \qquad \Psi^*p = p.    
\end{equation*} 
The pressure function $p$ is assumed to have toroidally nested level sets and the critical set of $p$, known as the axis of $p$ (which is an embedded circle), lies in the plane of reflection $\{z=0\}$. In this case, the field-lines are also all closed. In relation to Grad's Conjecture, Salat calculated that up to order 7 in a Fourier expansion, that his configurations are necessarily axisymmetric.

The equation
\begin{equation*}
\Psi_* B = -B    
\end{equation*}
on compact $M$ can be equivalently written in terms of the global flow $\psi^B$ of $B$ (which exists due to the condition $B \cdot n = 0$ on $\partial M$) as the statement that, for $t \in \RR$,
\begin{equation*}
\psi^B_{-t} = \Psi \circ \psi^B_t \circ \Psi^{-1}.  
\end{equation*}
If $\Psi \circ \Psi = \id$, then the flow $\psi^B$ is considered to be reversible in the same sense as \cite{Birkhoff_1917,price1935reversible,Wiggins_2003} and periodic orbits can be found for these systems by fixed points of the time-reversing symmetry $\Psi$. This is what takes place for the cases of Lortz and Salat mentioned above.

Given the interest in MHD equilibria with reflection symmetry, it is interesting to see what happens under, say, the symmetry $(x,y,z) \mapsto (-x,-y,-z)$, as this too is an orientation-reversing isometry of $\RR^3$. Unlike the aformentioned situations considered by Lortz and Salat, the only fixed point of this symmetry is the origin. Hence, it is less obvious in this case whether one also necessarily obtains periodic orbits for MHD equilibria which are symmetric under this transformation. Nevertheless, it is shown here that periodic orbits are abundant in this situation as well. 

To arrive at a precise statement about periodic orbits of MHD equilibria, we must first introduce a convenient index describing how a diffeomorphism of $\RR^3$ alters the poloidal and toroidal curves of an invariant embedded 2-torus.

\begin{proposition}\label{prop:homology-action-characterisation}
Let $S \subset \RR^3$ be an embedded 2-torus. Then, $\RR^3 \setminus S$ has two connected components, one of which is bounded and the other unbounded. Let $\cB_S$ denote the bounded component and $\cE_S$ denote the unbounded component. The closures $\overline{\cB_S} = \cB_S \cup \Bd(\cB_S)$ and $\overline{\cE_S} = \cE_S \cup \Bd(\cE_S)$ are smooth regular domains with boundary 
\begin{equation*}
\partial \overline{\cB_S} = \Bd(\cB_S) = S =  \Bd(\cE_S) = \partial \overline{\cE_S}.
\end{equation*}
Consider the inclusions $\jmath_{\cB_S} : S \subset \overline{\cB_S}$ and $\jmath_{\cE_S} : S \subset \overline{\cE_S}$ and the associated homomorphisms
\begin{align*}
{\jmath_{\cB_S}}_* &: H_1(S) \to H_1(\overline{\cB_S})\\
{\jmath_{\cE_S}}_* &: H_1(S) \to H_1(\overline{\cE_S}).
\end{align*}
Then, we have an internal direct sum
\begin{equation}\label{eq:split-poloidal-and-toroidal}
H_1(S) = \cP_S \oplus \cT_S,
\end{equation}
where $\cP_S$ and $\cT_S$ respectively denote, the \emph{poloidal subgroup} and the \emph{toroidal subgroup} of $H_1(S)$ given by
\begin{equation*}
\cP_S \coloneqq \ker {\jmath_{\cB_S}}_* \cong \ZZ, \qquad \cT_S \coloneqq \ker {\jmath_{\cE_S}}_* \cong \ZZ.
\end{equation*}
Let $\cI : \RR^3 \to \RR^3$ be a diffeomorphism such that $\cI(S) = S$ and let $\psi : S \to S$ denote the associated diffeomorphism on $S$. Then, the map $\psi_* : H_1(S) \to H_1(S)$ acts on the poloidal and toroidal subgroups as
\begin{equation}\label{eq:def-of-indices}
\begin{split}
\psi_*|_{\cP_S} &= \sigma_1 \id|_{\cP_S}\\
\psi_*|_{\cT_S} &= \sigma_2 \id|_{\cT_S},
\end{split}
\end{equation}
for some unique $\sigma_1,\sigma_2 \in \{+1,-1\}$. If $\cI$ is orientation-preserving, then $\sigma_1 = \sigma_2$. Otherwise, $\sigma_1 = -\sigma_2$.
\end{proposition}

Proposition \ref{prop:homology-action-characterisation} is a simple consequence of a Mayer-Vietoris sequence derived from the Jordan-Brower Separation Theorem (see \cite{GP,Lima}) applied to the 2-torus. See Section \ref{sec:poloidal-and-toroidal-homologies} for details.

\begin{remark}
The namesake of $\cP_S$ and $\cT_S$ come from the fact that the maps
\begin{align*}
{\jmath_{\cE_S}}_*|_{\cP_S} &: \cP_S \to H_1(\overline{\cE_S})\\
{\jmath_{\cB_S}}_*|_{\cT_S} &: \cT_S \to H_1(\overline{\cB_S}),
\end{align*}
are in fact isomorphisms. See Lemma \ref{lem:isomorphism-via-inclusions} in Section \ref{sec:poloidal-and-toroidal-homologies}.
\end{remark}

We now formally introduce the poloidal and toroidal index coming about from Proposition \ref{prop:homology-action-characterisation}.

\begin{definition}\label{def:pol-tor-index}
Let $S \subset \RR^3$ be an embedded 2-torus and $\cI : \RR^3 \to \RR^3$ be a diffeomorphism such that $\cI(S) = S$. We call the numbers $\sigma_1$ and $\sigma_2$ satisfying Equation \eqref{eq:def-of-indices} in Proposition \ref{prop:homology-action-characterisation} respectively the \emph{poloidal} and \emph{toroidal} \emph{index of $\cI$ on $S$}.
\end{definition}

Before stating our main result, we make a convenient definition to shorten the description of some simple vector fields.

\begin{definition}
Let $X$ be a vector field on a $2$-torus $S \subset \RR^3$ embedded in $\RR^3$. We will say that the orbits of $X$ are \textit{purely poloidal} (respectively, \textit{toroidal}), if for any (maximal) integral curve $\gamma : \RR \to S$ of $X$, $\gamma$ is a non-trivial periodic curve and, if $T > 0$ denotes its (minimal) period, then the closed curve $c : [0,1] \to S$ given by
\begin{equation*}
c(t) = \gamma(Tt),    
\end{equation*}
is such that $[c]$ is a generator of $\cP_S$ (respectively, $\cT_S$).
\end{definition}

We now state our main result concerning periodic orbits of MHD equilibria.

\begin{theorem}\label{thm:main}
Let $\cI : \RR^3 \to \RR^3$ be an orientation-reversing isometry. Let $M \subset \RR^3$ be a solid toroidal domain. Suppose that $\cI(M) = M$ and let $(B,p)$ be an MHD equilibrium on $M$ such that
\begin{equation}\label{eq:psi-action-on-B-intro}
\Psi_*B = \sigma B,
\end{equation}
for some $\sigma \in \{+1,-1\}$, where $\Psi : M \to M$ is diffeomorphism on $M$ induced by $\cI$. We have that $\cI(\partial M) = \partial M$ so we may let $\sigma_1$ and $\sigma_2$ denote the poloidal and toroidal index of $\cI$ on $\partial M$. Suppose that $p$ has toroidally nested surfaces. Let $y \in \RR$ be a regular value of $p$ in the image of $p$. Then, $B$ and $J = \curl B$ are tangent to the 2-torus $S_y = p^{-1}(y)$ and thus induce vector fields $B_y$ and $J_y$ on $S_y$. The following holds.
\begin{enumerate}
    \item If $\sigma = \sigma_1$, then the orbits of $B_y$ are purely poloidal and the orbits of $J_y$ are purely toroidal.
    \item If $\sigma = \sigma_2$, then the orbits of $B_y$ are purely toroidal and the orbits of $J_y$ are purely poloidal.
\end{enumerate}
\end{theorem}

A novelty of Theorem \ref{thm:main} is that we do not assume that $\Psi \circ \Psi = \id$, and so the theory of reversible dynamical systems cannot be used directly to prove it. Instead, the proof of Theorem \ref{thm:main} only depends on the broad dynamical properties of MHD equilbria. This is evident from the fact that it is derived from the following more general result.

\begin{proposition}\label{prop:periodic-orbits}
Let $S \subset \RR^3$ be a 2-torus which is embedded in $\RR^3$. Let $\cI : \RR^3 \to \RR^3$ be an orientation-reversing diffeomorphism such that $\cI(S) = S$. Let $X$ be an area-preserving vector field on $S$; that is, $X$ satisfies that
\begin{equation*}
\cL_X \mu = 0,    
\end{equation*}
for some area-form $\mu \in \Omega^2(S)$. Assume that $X$ is non-vanishing and that, writing $\psi$ for the diffeomorphism on $S$ induced by $\cI$,
\begin{equation*}
\psi_*X = fX,    
\end{equation*}
for some smooth function $f \in C^{\infty}(S)$. Then, $f$ is either everywhere positive or everywhere negative. Let $\sigma_1,\sigma_2$ be the poloidal and toroidal index of $\cI$ on $S$. Then, the following holds.
\begin{enumerate}
    \item If $\sign(f) = \sigma_1$, then the orbits of $X$ are purely poloidal.
    \item If $\sign(f) = \sigma_2$, then the orbits of $X$ are purely toroidal.
\end{enumerate}
\end{proposition}

Proposition \ref{prop:periodic-orbits} also has consequences in searching for well-behaved vacuum fields in the context of plasma physics. A vacuum field on $M$ is a smooth vector field $h$ satisfying
\begin{equation}\label{eq:harmonic-equation}
\curl h = 0, \qquad \Div h = 0, \qquad h \cdot n = 0.    
\end{equation}
As well-known \cite{Pfefferle_Noakes_Perrella_2021}, if $M$ is diffeomorphic to a solid torus, the vector space of solutions to Equation \eqref{eq:harmonic-equation} is 1-dimensional. We have the following result about the boundary behaviour of vacuum fields on solid tori.

\begin{corollary}\label{cor:vac-is-periodic}
Let $M \subset \RR^3$ be a solid toroidal domain and let $\cI : \RR^3 \to \RR^3$ be an orientation-reversing isometry such that $\cI(M) = M$. Let $h$ be a non-trivial solution to Equation \eqref{eq:harmonic-equation}. Consider the vector field $X$ on $\partial M$ induced by $h$ via the inclusion $\partial M \subset M$ (which exists because $h \cdot n = 0$). Then, either $X$ is not area-preserving on $\partial M$, or the orbits of $X$ are purely toroidal. 
\end{corollary}

In particular, in Corollary \ref{cor:vac-is-periodic}, there cannot exist a smooth $f > 0$ and a diffeomorphism $(\theta,\varphi) : S \to (\RR/2\pi\ZZ)^2$ for which
\begin{equation*}
X/f = a\partial_{\theta} + b\partial_\varphi,     
\end{equation*}
with $a,b \in \RR$ rationally independent. In terms of plasma physics terminology, $S = \partial M$ cannot be an irrational surface for $h$. 

As a by-product, Proposition \ref{prop:homology-action-characterisation} also derives a certain basic rigidity result of MHD equilibria and more general vector fields on solid toroidal domains.

\begin{proposition}\label{prop:basic-rigidity}
Let $M \subset \RR^3$ be a solid toroidal domain. Suppose that $B$ is a vector field on $M$ and $\rho$ is a smooth function on $M$ such that
\begin{equation*}
\Div B = 0, \qquad B \cdot \nabla \rho = 0 = \curl B \cdot \nabla \rho.    
\end{equation*}
Suppose further that $\rho$ has toroidally nested level sets, $\cI : \RR^3 \to \RR^3$ is an isometry such that $\cI(M) = M$, and
\begin{equation*}
\Psi^*\rho = \rho,    
\end{equation*}
where $\Psi : M \to M$ denotes the corresponding isometry of $M$. Then, the following holds.
\begin{enumerate}
    \item If $\cI$ is orientation-preserving, then
    \begin{equation*}
    \Psi_*B = \sigma B,    
    \end{equation*}
    where $\sigma = \sigma_1 = \sigma_2$ is the mutual poloidal and toroidal index of $\cI$ on $\partial M$.
    \item If $\cI$ is orientation-reversing, then
    \begin{equation*}
    (\Psi\circ \Psi)_* B = B.    
    \end{equation*}
\end{enumerate}
\end{proposition}

Proposition \ref{prop:basic-rigidity} couples with the fact that the toroidal index can be easily computed in a number of examples, as is discussed in Section \ref{sec:example-comp}. For instance, Proposition \ref{prop:basic-rigidity} implies a direct rigidity statement about Stellarator symmetry \cite{Dewar_Hudson_1998}. Stellarator symmetry, up to rotations and translations, is symmetry with respect to the isometry of $\RR^3$ given by $(x,y,z) \mapsto (x,-y,-z)$. The precise rigidity statement is the following.

\begin{corollary}\label{cor:Stellarator-symmetry}
Let $\cZ = \myspan\{(0,0,1)\}$ denote the $z$-axis of $\RR^3$. Let $M$ be a solid torus in $\RR^3$ which wraps around $\cZ$ once, in the sense that
\begin{enumerate}
    \item $M \cap \cZ = \emptyset$ and
    \item letting $\iota : M \subset \RR^3 \setminus \cZ$ denote the inclusion, $\iota_* : H_1(M) \to H_1(\RR^3 \setminus \cZ)$ is surjective.
\end{enumerate}
Let $\cI : \RR^3 \to \RR^3$ be the orientation-preserving isometry given by
\begin{equation*}
\cI(x,y,z) = (x,-y,-z)  
\end{equation*}
and suppose that $\cI(M) = M$. Let $\Psi : M \to M$ denote the corresponding diffeomorphism on $M$. Then, for any MHD equilibrium $(B,p)$ on $M$ where $p$ has toroidally nested level sets, one has the equivalence
\begin{equation}\label{eq:stell-sym}
p \circ \Psi = p \iff \Psi_* B = - B.  
\end{equation}
\end{corollary}

\begin{remark}
In Corollary \ref{cor:Stellarator-symmetry}, the implication $\Leftarrow$ is well-known in plasma physics and is independent of the toroidally nested property (see \cite[Appendix B.3]{Pfefferle_2015}). If $x,y,z  :M \to \RR$ denote the Cartesian coordinates on $M$, and $B = B^1\partial_x + B^2\partial_y + B^3\partial_z$, then Equation \eqref{eq:stell-sym} reads that
\begin{align*}
p(x,-y,-z) &= p(x,y,z)\\
&\Updownarrow\\
B^1(x,-y,-z) &= B^1(x,y,z),\\
B^2(x,-y,-z) &= -B^2(x,y,z),\\
B^3(x,-y,-z) &= -B^3(x,y,z).
\end{align*}
\end{remark}

As a final comment, Theorem \ref{thm:main} does not mention the behaviours of orbits near the axis of the MHD equilibrium $(B,p)$. It is plausible that degenerate behaviour may occur in some cases. For instance, perhaps on the axis of $p$, there could be some points at which $B$ is zero and other places where it is not. We show that this cannot happen if $p$ is assumed to be Morse-Bott.

\begin{proposition}\label{prop:nice-axis-behaviour}
Let $M$ be an oriented Riemannian 3-manifold with boundary and $(B,p)$ be an MHD equilibrium on $M$. Let $\gamma \subset \Int M$ be an embedded circle which is also a Morse-Bott critical submanifold for $p$. Then, setting $J = \curl B$, the following holds.
\begin{enumerate}
    \item If $p$ is locally maximised on $\gamma$, then either \label{item:phys-rel}
    \begin{enumerate}
    \item $B|_\gamma = 0$ and $J|_\gamma$ has no zeros, or
    \item $B|_\gamma$ has no zeros and $J|_\gamma = k B|_\gamma$ for some $k \in \RR$. \label{item:phys-rel-b}
    \end{enumerate}
    \item If $p$ is locally minimised on $\gamma$, then $B|_\gamma$ has no zeros and $J|_\gamma = k B|_\gamma$ for some $k \in \RR$.
\end{enumerate}     
\end{proposition}

Proposition \ref{prop:nice-axis-behaviour} follows from considering the Laplacian $\Delta p$ on $\gamma$, leveraging also the fact that $[J,B] = 0$, where $[\cdot,\cdot]$ denotes the Lie bracket. The most physically relevant case for magnetic confinement fusion devices \cite{Ongena_Koch_Wolf_Zohm_2016} in Proposition \ref{prop:nice-axis-behaviour} is case \ref{item:phys-rel-b}.

The remainder of this paper is structured as follows. In Section \ref{sec:poloidal-and-toroidal-homologies}, we review the Jordan-Brower Separation Theorem and a homological splitting which includes that of Equation \eqref{eq:split-poloidal-and-toroidal}. This is used to prove Proposition \ref{prop:homology-action-characterisation}. In Section \ref{sec:dyn-prop-and-cor}, we prove Proposition \ref{prop:periodic-orbits} and Corollary \ref{cor:vac-is-periodic}. After this, in Section \ref{sec:extension-to-nested-toroidal-surfaces}, we revisit the homology construction from Section \ref{sec:poloidal-and-toroidal-homologies} and show that the poloidal and toroidal indices of a diffeomorphism on invariant 2-tori in a trivial foliation do not vary across the tori. This enables a proof of Proposition \ref{prop:basic-rigidity}. Then, in Section \ref{sec:proof-main-res} we prove the main result. In Section \ref{sec:example-comp} we provide some examples in which the toroidal index of a diffeomorphism can be easily computed, and prove Corollary \ref{cor:Stellarator-symmetry}. In Section \ref{sec:magnetic axes}, we prove Proposition \ref{prop:nice-axis-behaviour}. Lastly, in Section \ref{sec:sympectic-algebra}, we mention a generalisation of Proposition \ref{prop:periodic-orbits} in the setting of abstract 2-tori.

\section{Action on homology by a diffeomorphism in \texorpdfstring{$\RR^n$}{Rn} on an invariant hypersurface}\label{sec:poloidal-and-toroidal-homologies}

The main purpose of this section is to prove Proposition \ref{prop:homology-action-characterisation}. The exposition will be more general than the 2-torus. To begin, we state a detailed version of the Jordan-Brower Separation Theorem (see \cite{GP,Lima}).

\begin{theorem}[Jordan-Brower Separation]\label{thm:detailed-separation}
Let $S$ be a closed connected codimension 1 submanifold $S$ of $\RR^n$. Then $\RR^n \setminus S$ has two connected components. One of the components is a bounded open set, the other is an unbounded open set, and both of their topological boundaries are $S$. Letting $\cB_S$ and $\cE_S$ denote respectively the bounded and unbounded components, there exists a map $\Sigma : S \times (-1,1) \to \RR^n$ which is a diffeomorphism onto its image such that $\Sigma(x,0) = x$ for $x \in S$ and
\begin{equation*}
\Sigma(S \times (-1,0)) \subset \cB_S, \qquad \Sigma(S \times (0,1)) \subset \cE_S.
\end{equation*}
In particular, $S$ is orientable and the closures $\overline{\cB_S} = \cB_S \cup \Bd(\cB_S)$ and $\overline{\cE_S} = \cE_S \cup \Bd(\cE_S)$ are smooth regular domains with boundary 
\begin{equation*}
\partial \overline{\cB_S} = \Bd(\cB_S) = S = \Bd(\cE_S) = \partial \overline{\cE_S}.
\end{equation*}
\end{theorem}

Theorem \ref{thm:detailed-separation} yields some invariants of embedded hypersurfaces under diffeomorphism.

\begin{restatable}{lemma}{homgroupinvar}\label{lem:homolgy-group-invariants} 
Let $S$ be a closed and connected submanifold of $\RR^n$ of codimension 1. Let $\cI : \RR^n \to \RR^n$ be a diffeomorphism and set $\cI(S) = S'$. Let $\cB_S$, $\cE_S$, $\cB_{S'}$, and $\cE_{S'}$ be as in Theorem \ref{thm:detailed-separation} for the submanifolds $S$ and $S'$. Then
\begin{equation}\label{eq:diffeo-preserves-int-and-ext}
\begin{split}
\cI(\overline{\cB_S}) &= \overline{\cB_{S'}},\\
\cI(\overline{\cE_S}) &= \overline{\cE_{S'}}.
\end{split}
\end{equation}
Let $\psi : S \to S'$ denote the diffeomorphism induced from $\cI$. Then, $\cI$ is orientation-preserving if and only if $\psi$ is orientation-preserving. Consider the inclusions
\begin{align*}
\jmath_{\cB_S} &: S \subset \overline{\cB_S}, & \jmath_{\cE_S} &: S \subset \overline{\cE_S},\\
\jmath_{\cB_{S'}} &: S' \subset \overline{\cB_{S'}}, & \jmath_{\cE_{S'}} &: S' \subset \overline{\cE_{S'}},
\end{align*}
and the associated homomorphisms
\begin{align*}
{\jmath_{\cB_S}}_* &: H_k(S) \to H_k(\overline{\cB_S}), & {\jmath_{\cE_S}}_* &: H_k(S) \to H_k(\overline{\cE_S}),\\
{\jmath_{\cB_{S'}}}_* &: H_k(S') \to H_k(\overline{\cB_{S'}}), & {\jmath_{\cE_{S'}}}_* &: H_k(S') \to H_k(\overline{\cE_{S'}}),
\end{align*}
where $k \in \NN_0 = \{0,1,...\}$. Then, $\psi_* : H_k(S) \to H_k(S')$ satisfies
\begin{equation}\label{eq:diffeo-preserves-int-and-ext-homology}
\begin{split}
\psi_*(\ker {\jmath_{\cB_S}}_*) &= \ker {\jmath_{\cB_{S'}}}_*,\\
\psi_*(\ker {\jmath_{\cE_S}}_*) &= \ker {\jmath_{\cE_{S'}}}_*.
\end{split}
\end{equation}
\end{restatable}

For a proof of Lemma \ref{lem:homolgy-group-invariants}, see Appendix \ref{app:deferred-proofs-1}. Complementary to Lemma \ref{lem:homolgy-group-invariants}, there is the following well-known result (see \cite{Cantarella02,PerrellaThesis}), whose proof is presented here for completeness.

\begin{lemma}\label{lem:isomorphism-via-inclusions}
Let $S$ be closed connected codimension 1 submanifold $S$ of $\RR^n$. Consider the inclusions $\jmath_{\cB_S} : S \subset \overline{\cB_S}$ and $\jmath_{\cE_S} : S \subset \overline{\cE_S}$ and the associated homomorphisms
\begin{align*}
{\jmath_{\cB_S}}_* &: H_k(S) \to H_k(\overline{\cB_S}),\\
{\jmath_{\cE_S}}_* &: H_k(S) \to H_k(\overline{\cE_S}),
\end{align*}
where $k \in \NN = \{1,2,...\}$. Then, the map
\begin{equation*}
({\jmath_{\cB_S}}_*,{\jmath_{\cE_S}}_*) : H_k(S) \to H_k(\overline{\cB_S}) \oplus H_k(\overline{\cE_S}),    
\end{equation*}
is a group isomorphism. We have an internal direct sum
\begin{equation*}
H_k(S) = \ker {\jmath_{\cB_S}}_* \oplus \ker {\jmath_{\cE_S}}_*.
\end{equation*}
Moreover, the restricted maps
\begin{align*}
{\jmath_{\cB_S}}_*|_{\ker {\jmath_{\cE_S}}_*} &: {\ker {\jmath_{\cE_S}}_*} \to H_k(\overline{\cB_S}),\\
{\jmath_{\cE_S}}_*|_{\ker {\jmath_{\cB_S}}_*} &: {\ker {\jmath_{\cB_S}}_*} \to H_k(\overline{\cE_S}),
\end{align*}
are also isomorphisms.
\end{lemma}

\begin{proof}
For the first claim, we follow \cite[Proposition 2.4.2]{PerrellaThesis} and \cite{Cantarella02} with some more details filled-in. We work in the topological category for convenience. This means that all maps and retractions considered here are continuous but not necessarily smooth. We first work with a Mayer-Vietoris sequence as follows. Let $\Sigma : S \times (-1,1) \to \RR^n$ be a map as in Theorem \ref{thm:detailed-separation}. Let $C = \Sigma(S \times (-1,0])$ and consider the open set
\begin{equation*}
N = \overline{\cE_S} \cup C
\end{equation*}
of $\RR^n$. The interiors of $\overline{\cB_S}$ and $N$ in $\RR^n$ cover $\RR^n$ and so we may apply the Mayer-Vietoris sequence to $\overline{B_S}$ and $N$. Part of the sequence reads
\begin{equation*}
\cdots \to H_{k+1}(\RR^n) \to H_k(\overline{\cB_S} \cap N) \to H_k(\overline{\cB_S}) \oplus H_k(N) \to H_k(\RR^n) \to \cdots.
\end{equation*}
Now, because $k \geq 1$, $H_{k+1}(\RR^n)$ and $H_k(\RR^n)$ are trivial groups. Also, from the map $\Sigma$, we see that $C$ deformation retracts to $S$. Hence, the above part of the sequence reduces to the short exact sequence
\begin{equation*}
0 \to H_k(S) \to H_k(\overline{\cB_S}) \oplus H_k(N) \to 0,
\end{equation*}
where the map $H_1(S) \to H_1(\overline{\cB_S}) \oplus H_1(N)$ is $({\jmath_{\cB_S}}_*,i_*)$ where $i : S \subset N$ is the inclusion. In particular, $({\jmath_{\cB_S}}_*,i_*)$ is a group isomorphism. In addition, $N$ deformation retracts to $\overline{\cE_S}$ and so, we obtain the short exact sequence
\begin{equation*}
0 \to H_k(S) \to H_k(\overline{\cB_S}) \oplus H_k(\overline{\cE_S}) \to 0,
\end{equation*}
where the map $H_k(S) \to H_k(\overline{\cB_S}) \oplus H_k(\overline{\cE_S})$ is $({\jmath_{\cB_S}}_*,{\jmath_{\cE_S}}_*)$. Hence, $({\jmath_{\cB_S}}_*,{\jmath_{\cE_S}}_*)$ is an isomorphism.

In particular, because the direct sum $H_k(\overline{\cB_S}) \oplus H_k(\overline{\cE_S})$ can be expressed internally as
\begin{equation*}
H_k(\overline{\cB_S}) \oplus H_k(\overline{\cE_S}) = ( \{0\} \times H_k(\overline{\cE_S})) \oplus (H_k(\overline{\cB_S}) \times \{0\}),
\end{equation*}
the isomorphism $({\jmath_{\cB_S}}_*,{\jmath_{\cE_S}}_*)$ gives the internal direct sum
\begin{equation}\label{eq:kernel-direct-sum}
H_k(S) = \ker {\jmath_{\cB_S}}_* \oplus \ker {\jmath_{\cE_S}}_*.
\end{equation}
In particular, Equation \eqref{eq:kernel-direct-sum} says that $\ker {\jmath_{\cB_S}}_* \cap \ker {\jmath_{\cE_S}}_* = \{0\}$ and so both maps
\begin{align*}
{\jmath_{\cB_S}}_*|_{\ker {\jmath_{\cE_S}}_*} &: {\ker {\jmath_{\cE_S}}_*} \to H_k(\overline{\cB_S}),\\
{\jmath_{\cE_S}}_*|_{\ker {\jmath_{\cB_S}}_*} &: {\ker {\jmath_{\cB_S}}_*} \to H_k(\overline{\cE_S}),
\end{align*}
are injective. Surjectivity of $({\jmath_{\cB_S}}_*,{\jmath_{\cE_S}}_*)$ implies surjectivity of ${\jmath_{\cB_S}}_*$ and ${\jmath_{\cE_S}}_*$. Equation \eqref{eq:kernel-direct-sum} therefore implies surjectivity of the restricted maps
\begin{align*}
{\jmath_{\cB_S}}_*|_{\ker {\jmath_{\cE_S}}_*} &: {\ker {\jmath_{\cE_S}}_*} \to H_k(\overline{\cB_S}),\\
{\jmath_{\cE_S}}_*|_{\ker {\jmath_{\cB_S}}_*} &: {\ker {\jmath_{\cB_S}}_*} \to H_k(\overline{\cE_S}),
\end{align*}
and so, in conclusion, both maps are isomorphisms.
\end{proof}

The last ingredient needed to prove Proposition \ref{prop:homology-action-characterisation} is counting the rank of the groups $\ker {\jmath_{\cB_S}}_*$ and $\ker {\jmath_{\cE_S}}_*$. We present a method for counting the ranks which first goes through cohomology. We do so because of the interesting connection with Hodge Theory.

\begin{lemma}
Let $M$ be a compact orientable manifold with boundary $\partial M$. Let $\dim M = n$. Let $\kappa \in \Omega^k(\partial M)$ be a closed $k$-form on $\partial M$. Let $\jmath : \partial M \subset M$ be the inclusion. Then, there exists a $k$-form $\omega \in \Omega^k(M)$ such that
\begin{equation}\label{eq:Dirichlet-for-d}
d\omega = 0, \qquad \jmath^*\omega = \kappa
\end{equation}
if and only if, for all closed $(n-k-1)$-forms $\eta \in \Omega^{n-k-1}(M)$,
\begin{equation}\label{eq:coisotropic}
\int_{\partial M} \kappa \wedge \jmath^*\eta = 0.
\end{equation}
\end{lemma}

\begin{proof}
For the forward direction, assume that there exists a solution $\omega \in \Omega^k(M)$ to Equation \eqref{eq:Dirichlet-for-d}. Then, for all closed $(n-k)$-forms $\eta \in \Omega^{n-k}(M)$,
\begin{align*}
\int_{\partial M} \kappa \wedge \jmath^*\eta &= \int_{\partial M} \jmath^*\omega \wedge \jmath^*\eta\\
&= \int_{\partial M} \jmath^*(\omega \wedge \eta)\\
&= \int_{M} d(\omega \wedge \eta)\\
&= 0,
\end{align*}
where in the last line, we used the fact that $d\omega = 0$ and $d \eta = 0$.

The reverse direction is essentially that of \cite[Theorem 3.1.1]{Schwarz} which makes use of a Hodge decomposition theorem. Fix a Riemannian metric $g$ on $M$. For all $p \in \{0,...,n\}$, let $\star : \Omega^{p}(M) \to \Omega^{n-p}(M)$ be the induced Hodge Star operator and $\delta = (-1)^{np+n+1} \star d \star$ be the induced co-differential operator $\Omega^p(M) \to \Omega^{p-1}(M)$. Then with respect to the induced $L^2$ inner product $\langle,\rangle_{L^2}$ on $\Omega^k(M)$ given by
\begin{equation*}
\langle \omega, \eta \rangle = \int_M \omega \wedge \star \eta,
\end{equation*}
one has the $L^2$-orthogonal Hodge decomposition \cite{Schwarz}
\begin{equation}\label{eq:the-decomp}
\Omega^k(M) = d\Omega^{k-1}_D(M) \oplus \delta \Omega^{k+1}_N(M) \oplus \cH^k(M),     
\end{equation}
where for any $p \in \{0,...,n\}$,
\begin{align*}
\Omega^{p}_D(M) &= \{\omega \in \Omega^p(M) : \jmath^*\omega = 0\},\\
\Omega^p_N(M) &= \{\omega \in \Omega^p(M) : \jmath^*(\star \omega) = 0\},\\   
\cH^p(M) &= \{\omega \in \Omega^p(M) : d\omega = 0, \, \delta \omega = 0\}.  
\end{align*}

Now, let $\Lambda \in \Omega^k(M)$ be such that $\jmath^*\Lambda = \lambda$. Such a form $\Lambda$ exists from the Collar Neighborhood Theorem \cite[Theorem 9.25]{Lee12}, that is, there exists a diffeomorphism $\Sigma : \partial M \times [0,1) \to U$ where $U$ is a neighborhood of $\partial M$ in $M$ and $\Sigma(x,0) = x$ for all $x \in \partial M$. Set $\eta = d\Lambda$. Then, from Equation \eqref{eq:the-decomp}, we have the orthogonal decomposition
\begin{equation*}
\eta = d\alpha_\eta  + \delta \beta_\eta + \nu_\eta.
\end{equation*}
Because $\eta = d\Lambda$ is exact, one has that $\langle\eta,\delta \beta\rangle_{L^2} = 0$ for all $\delta\beta \in \delta \Omega^{k+1}_N(M)$. Hence, $\delta \beta_\eta = 0$ and so
\begin{equation*}
\eta = d\alpha_\eta + \nu_\eta.
\end{equation*}
Now we set
\begin{equation*}
\omega = \Lambda - \alpha_\eta.    
\end{equation*}
Then, we have
\begin{gather*}
\jmath^*\omega = \jmath^*\Lambda-\jmath^*\alpha_\eta = \lambda,\\
d\omega = d\Lambda - d\alpha_\eta = \eta - d\alpha_\eta = \nu_\eta.
\end{gather*}
Lastly, for any $\nu \in \cH^{k+1}(M)$, because $d (\star \nu) = 0$, one finds that
\begin{align*}
(d\omega,\nu)_{L^2} &= \int_M d\omega \wedge \star \nu \\
&= \int_M d(\omega \wedge \star \nu)\\
&= \int_{\partial M} \jmath^*\omega \wedge \jmath^*(\star \nu)\\
&= \int_{\partial M} \lambda \wedge \jmath^*(\star \nu)\\
&= 0,
\end{align*}
where in the last line, we used the fact that Equation \eqref{eq:coisotropic} holds. Hence, because $\nu_\eta \in \cH^{k+1}(M)$, it follows that $\nu_\eta = 0$. Therefore, $\omega$ solves Equation \eqref{eq:Dirichlet-for-d}.
\end{proof}

We obtain, as a corollary, the following well-known result.

\begin{corollary}\label{cor:lagrangian-subspace}
Let $M$ be a compact orientable manifold with boundary $\partial M$. Suppose that $\dim M = 3 + 4m$ for some $m \in \mathbb{N}_0 = \{0\} \cup \NN$. Let $\jmath : \partial M \subset M$ be the inclusion, $k = 2m+1$, and consider the pullback 
\begin{equation*}
\jmath^* : H^{k}_{\text{dR}}(M) \to H^{k}_{\text{dR}}(\partial M)    
\end{equation*}
on de Rham cohomology. Then, by Poincar\'e duality, the bilinear map $W : H^k_{\text{dR}}(\partial M) \times H^k_{\text{dR}}(\partial M) \to \RR$ given by
\begin{equation*}
W([\kappa],[\lambda]) = \int_{\partial M} \kappa \wedge \lambda
\end{equation*}
is a symplectic form, and the image
\begin{equation*}
    \jmath^*(H^k_{\text{dR}}(M)) \subset H^k_{\text{dR}}(\partial M)
\end{equation*}
is a Lagrangian subspace for $W$.
\end{corollary}

The ranks of $\ker {\jmath_{\cB_S}}_*$ and $\ker {\jmath_{\cE_S}}_*$ can now be computed as follows.

\begin{lemma}\label{lem:ranks-of-homology-subgroups}
Let $S$ be closed connected codimension 1 submanifold $S$ of $\RR^n$. Assume that $n = 3 + 4m$ for some $m \in \mathbb{N}_0$. Consider the inclusions $\jmath_{\cB_S} : S \subset \overline{\cB_S}$ and $\jmath_{\cE_S} : S \subset \overline{\cE_S}$ and the associated homomorphisms
\begin{align*}
{\jmath_{\cB_S}}_* &: H_k(S) \to H_k(\overline{\cB_S}),\\
{\jmath_{\cE_S}}_* &: H_k(S) \to H_k(\overline{\cE_S}),
\end{align*}
where $k = 2m + 1$. Then, 
\begin{equation*}
\rank{\ker {\jmath_{\cB_S}}_*} = \rank{H_k(\overline{\cB_S})} = \frac{\rank{H_k(S)}}{2} = \rank{H_k(\overline{\cE_S})} = \rank{\ker {\jmath_{\cE_S}}_*}.
\end{equation*}
\end{lemma}

\begin{proof}
Firstly, we have from Corollary \ref{cor:lagrangian-subspace} that
\begin{equation*}
\jmath_{\cB_S}^*(H^k_{\text{dR}}(\overline{\cB_S})) \subset H^k_{\text{dR}}(S)
\end{equation*}
is a Lagrangian subspace for the symplectic form $W : H^k_{\text{dR}}(S) \times H^k_{\text{dR}}(S) \to \RR$ given by 
\begin{equation*}
W([\kappa],[\lambda]) = \int_{S} \kappa \wedge \lambda,
\end{equation*}
for closed $\kappa,\lambda \in H^k_{\text{dR}}(\overline{\cB_S})$. Hence, writing $b_k = \dim H^k_{\text{dR}}(S)$, we have that $b_k$ is even and
\begin{equation*}
\dim{\jmath_{\cB_S}^*(H^k_{\text{dR}}(\overline{\cB_S}))} = \frac{b_k}{2}.
\end{equation*}

The map ${\jmath_{\cB_S}}_* : H_k(S) \to H_k(\overline{\cB_S})$ is surjective by Lemma \ref{lem:isomorphism-via-inclusions}. Thus, by de Rham's Theorem, $\jmath_{\cB_S}^* : H^k_{\text{dR}}(\overline{\cB_S}) \to H^k_{\text{dR}}(S)$ is injective. Hence, we can write that
\begin{equation*}
\dim{H^k_{\text{dR}}(\overline{\cB_S})} = \frac{b_k}{2}.
\end{equation*}
Therefore, by de Rham's Theorem, $\rank{H_k(S)} = b_k$ is even, and
\begin{equation}\label{eq:half-dim}
\rank{H_k(\overline{\cB_S})} = \frac{\rank{H_k(S)}}{2}.
\end{equation}
By Lemma \ref{lem:isomorphism-via-inclusions}, we have that
\begin{align*}
\rank {\ker {\jmath_{\cE_S}}_*} &= \rank{H_k(\overline{\cB_S})},\\
\rank{\ker {\jmath_{\cB_S}}_*} + \rank{ {\ker {\jmath_{\cE_S}}_*}} &= \rank{H_k(S)}.
\end{align*}
Comparing with Equation \eqref{eq:half-dim}, we see that
\begin{equation*}
\rank{\ker {\jmath_{\cB_S}}_*} = \frac{\rank{H_k(S)}}{2} = \rank{\ker {\jmath_{\cE_S}}_*}.
\end{equation*}
\end{proof}

We now have more than enough information to prove Proposition \ref{prop:homology-action-characterisation}.

\begin{proof}[Proof of Proposition \ref{prop:homology-action-characterisation}]
The proof is essentially just citing the appropriate lemmas. Let $S \subset \RR^3$ be an embedded 2-torus. From Theorem \ref{thm:detailed-separation}, $\RR^3 \setminus S$ has two connected components, one of which is bounded and the other unbounded. Let $\cB_S$ denote the bounded component and $\cE_S$ denote the unbounded component. The closures $\overline{\cB_S} = \cB_S \cup \Bd(\cB_S)$ and $\overline{\cE_S} = \cE_S \cup \Bd(\cE_S)$ are smooth regular domains with boundary 
\begin{equation*}
\partial \overline{\cB_S} = \Bd(\cB_S) = S =  \Bd(\cE_S) = \partial \overline{\cE_S}.
\end{equation*}
Consider the inclusions $\jmath_{\cB_S} : S \subset \overline{\cB_S}$ and $\jmath_{\cE_S} : S \subset \overline{\cE_S}$ and the associated homomorphisms
\begin{align*}
{\jmath_{\cB_S}}_* &: H_1(S) \to H_1(\overline{\cB_S})\\
{\jmath_{\cE_S}}_* &: H_1(S) \to H_1(\overline{\cE_S}).
\end{align*}
From Lemma \ref{lem:isomorphism-via-inclusions}, we have an internal direct sum
\begin{equation*}
H_1(S) = \cP_S \oplus \cT_S,
\end{equation*}
where $\cP_S$ and $\cT_S$ respectively denote, the poloidal and toroidal subgroups of $H_1(S)$ given by
\begin{equation*}
\cP_S \coloneqq \ker {\jmath_{\cB_S}}_* \cong \ZZ, \qquad \cT_S \coloneqq \ker {\jmath_{\cE_S}}_* \cong \ZZ,
\end{equation*}
where the ranks of $\cP_S$ and $\cT_S$ are deduced from Lemma \ref{lem:ranks-of-homology-subgroups}. Let $\cI : \RR^3 \to \RR^3$ be a diffemorphism such that $\cI(S) = S$ and let $\psi : S \to S$ denote the associated diffeomorphism on $S$. For the last claim, noting that if $\phi : \ZZ \to \ZZ$ is an isomorphism of groups, then necessarily $\phi = \pm \id$. Combining this with Lemma \ref{lem:isomorphism-via-inclusions}, the map $\psi_* : H_1(S) \to H_1(S)$ acts on the poloidal and toroidal subgroups as
\begin{align*}
\psi_*|_{\cP_S} &= \sigma_1 \id|_{\cP_S}\\
\psi_*|_{\cT_S} &= \sigma_2 \id|_{\cT_S},
\end{align*}
for some $\sigma_1,\sigma_2 \in \{+1,-1\}$. Again from Lemma \ref{lem:isomorphism-via-inclusions}, if $\cI$ is orientation-preserving, then $\sigma_1 = \sigma_2$. Otherwise, $\sigma_1 = -\sigma_2$.
\end{proof}

We end with a simple proposition which is helpful in understanding the toroidal index in the case of solid tori. This is used in sections \ref{sec:dyn-prop-and-cor} and \ref{sec:example-comp}.
\begin{proposition}\label{prop:toroidal-index-on-solid-torus}
Let $\cI : \RR^3 \to \RR^3$ be a diffeomorphism and let $M \subset \RR^3$ be an embedded solid torus. Assume that $\cI(M) = M$. Then, $S = \partial M$ is a 2-torus and $\cI(S) = S$. Let $\sigma_2$ denote the toroidal index of $\cI$ on $S$. Let $\Psi : M \to M$ denote the diffeomorphism on $M$ induced by $\cI$. Then, on the first homology $H_1(M)$ of $M$,
\begin{equation*}
\Psi_* = \sigma_2 \id.
\end{equation*}
\end{proposition}

\begin{proof}
First, observe that $\overline{\cB_S} = M$ where $\cB_S$ is the bounded connected component of $\RR^3 \setminus S$ as defined in Proposition \ref{prop:homology-action-characterisation}. Indeed, because $\Int M$ is connected, it is contained in exactly one connected component $C$ of $\RR^n \setminus S$. Moreover $\Int M = M \cap C$, and so $\Int M$ is both open and closed in $C$. Thus, $\Int M = C$. Because $\Int M$ is bounded, it follows that $\Int M = \cB_S$. Hence, because $\Int M$ is dense in $M$, $M = \overline{\Int M} = \overline{\cB_S}$, as claimed.

Now, by definition, letting $\psi : S \to S$ denote the diffeomorphism on $S$ inherited from $\cI$, 
\begin{equation}\label{eq:def-of-tor-index}
\psi_*|_{\cT_S} = \sigma_2 \id,
\end{equation}
where $\cT_s \subset H_1(S)$ is the toroidal subgroup of $S$. Let $\jmath : \partial M \subset M$ denote the inclusion. Then, from the relations
\begin{equation*}
\Psi \circ \jmath = \jmath \circ \psi,    
\end{equation*}
we derive on first homology that
\begin{equation}\label{eq:commutive-relation}
\Psi_* \circ \jmath_* = \jmath_* \circ \psi_*.    
\end{equation}
Equations \eqref{eq:def-of-tor-index} and \eqref{eq:commutive-relation} imply
\begin{align*}
\Psi_* \circ \jmath_*|_{\cT_S} &= \jmath_* \circ \psi_*|_{\cT_S}\\
&= \jmath_* \circ (\sigma_2 \id)|_{\cT_S}\\
&= \sigma_2~\jmath_*|_{\cT_S}.
\end{align*}
However, because $M = \overline{\cB_S}$, Lemma \ref{lem:isomorphism-via-inclusions} implies that the restricted map
\begin{equation*}
{\jmath_*}|_{\cT_S} : \cT_S \to H_1(M),
\end{equation*}
is an isomorphism. Hence,
\begin{equation*}
\Psi_* = \sigma_2 \id.
\end{equation*}
\end{proof}

\section{Proof of Proposition \ref{prop:periodic-orbits} and Corollary \ref{cor:vac-is-periodic}}\label{sec:dyn-prop-and-cor}

The heart of Theorem \ref{thm:main} is Proposition \ref{prop:periodic-orbits}, which we now prove.

\begin{proof}[Proof of Proposition \ref{prop:periodic-orbits}]
We note in the beginning that $f$ is either everywhere positive or everywhere negative because $fX = \psi_* X$ has no zeros on $S$ and $S$ is connected. Now, for the main part of the proof. By Kolmogorov's Theorem (see, for instance \cite{Kolmogorov53} and \cite[Theorem 1.1]{Sternberg}), there exists a diffeomorphism $(x,y) : S \to \Sone \times \Sone$, a positive smooth function $0 < h \in C^{\infty}(S)$, and a vector $(a,b) \in \RR^2 \setminus \{(0,0)\}$ such that
\begin{equation*}
X = h(a \partial_x + b \partial_y).
\end{equation*}

Consider the curves $c_1,c_2 : [0,1] \to \Sone \times \Sone$ given by
\begin{equation*}
c_1(t) = ([t],[0]), \quad c_2(t) = ([0],[t]),
\end{equation*}
where $[\cdot] : \RR \to \RR/\ZZ$ denotes the quotient projection. By modifying the diffeomorphism $(x,y)$ by an action of $GL(2,\ZZ)$, we may assume that there exist generators $\gamma_1 \in \cP_S$ and $\gamma_2 \in \cT_S$ of the poloidal and toroidal subgroups of $S$ such that, under the pushfoward of $(x,y) : S \to \Sone \times \Sone$, $\gamma_1$ and $\gamma_2$ are the homology classes $[c_1]$ and $[c_2]$ respectively.

Because $\sigma_1^2 = 1 = \sigma_2^2$, we therefore have in $H^1_{\text{dR}}(S)$ that
\begin{equation}\label{eq:action-on-cohomology}
[\psi^*dx] = \sigma_1 [dx], \qquad [\psi^*dy] = \sigma_2 [dy].    
\end{equation}
Now, consider the closed 1-form
\begin{equation*}
\omega = -bdx + ady.    
\end{equation*}
Then,
\begin{equation*}
\omega(X) = 0.    
\end{equation*}
Hence, because $\psi_* X = fX$, we have $\omega(\psi_* X) = 0$ so that
\begin{equation*}
(\psi^*\omega)(X) = \omega(\psi_*X) \circ \psi = 0. 
\end{equation*}
Thus,
\begin{equation}\label{eq:zero-wedge}
\omega \wedge \psi^*\omega = 0.    
\end{equation}
On the other hand, from Equation \eqref{eq:action-on-cohomology},
\begin{equation*}
[\psi^*\omega] = [-b\sigma_1 dx + a\sigma_2 dy],    
\end{equation*}
and so, from Equation \eqref{eq:zero-wedge}, in $H^2_{\text{dR}}(S)$,
\begin{align*}
0 &= [\omega \wedge \psi^*\omega]\\
&= [-bdx + ady] \wedge [-b\sigma_1 dx + a\sigma_2 dy]\\
&= ab(\sigma_1-\sigma_2)[dx \wedge dy].
\end{align*}
Hence, because $\sigma_1 \neq \sigma_2$ (because $\psi$ is orientation-reversing), 
\begin{equation*}
ab = 0.    
\end{equation*}

Now, suppose that $b = 0$. We have then that
\begin{equation*}
X = ha\partial_x.    
\end{equation*}
In particular, the orbits of $X$ are purely poloidal. Moreover,
\begin{equation*}
\psi_* (ha\partial_x) = \psi_* X = f X = fha\partial_x.     
\end{equation*}
So, on the one hand,
\begin{align*}
(\psi^*dx)(ha\partial_x) &= dx(\psi_* (ha\partial_x)) \circ \psi\\
&= dx(fha\partial_x) \circ \psi\\
&= a (f\circ \psi)(h \circ \psi). 
\end{align*}
On the other hand, using Equation \eqref{eq:action-on-cohomology}, for some $g \in C^{\infty}(S)$,
\begin{align*}
(\psi^*dx)(ha\partial_x) &= (\sigma_1 dx + dg)(ha\partial_x)\\
&= a \sigma_1 h + ah~dg(\partial_x).
\end{align*}
Taking some critical point of $p \in S$ of $g$ (which exists by compactness of $S$), we have at $p$ that
\begin{equation*}
a f(\psi(p)) h(\psi(p)) = a \sigma_1 h(p)
\end{equation*}
and because $h > 0$ and $a \neq 0$ (as $b = 0$), we can write that
\begin{equation*}
f(\psi(p))/\sigma_1 = h(p)/(h(\psi(p))) > 0.
\end{equation*}
Therefore, because $f$ has no zeros,
\begin{equation*}
\sign(f) = \sigma_1.
\end{equation*}
We have therefore shown that if $b = 0$, then the orbits of $X$ are purely poloidal and $\sign(f) = \sigma_1$.

The proof is similar to show that if $a = 0$, then the orbits of $X$ are purely toroidal and $\sign(f) = \sigma_2$. In summary,
\begin{enumerate}
    \item $a = 0$ or $b = 0$ but $(a,b) \in \RR^2 \setminus \{(0,0)\}$.
    \item If $b = 0$, then the orbits of $X$ are purely poloidal and $\sign(f) = \sigma_1$.
    \item If $a = 0$, then the orbits of $X$ are purely toroidal and $\sign(f) = \sigma_2$.
\end{enumerate}
Hence, if $\sign(f) = \sigma_1$, then $\sign(f) \neq \sigma_2$ and so $b = 0$, implying that the orbits of $X$ are purely poloidal. If $\sign(f) = \sigma_2$, then $\sign(f) \neq \sigma_1$ and so $a = 0$, implying that the orbits of $X$ are purely toroidal.
\end{proof}

From Proposition \ref{prop:periodic-orbits} we can immediately prove Corollary \ref{cor:vac-is-periodic}.

\begin{proof}[Proof of Corollary \ref{cor:vac-is-periodic}]
We recall the setting. Let $M \subset \RR^3$ be a solid toroidal domain and let $\cI : \RR^3 \to \RR^3$ be an orientation-reversing isometry such that $\cI(M) = M$. Let $h$ be a non-trivial solution to Equation \eqref{eq:harmonic-equation}. Let $X$ be the vector field on $\partial M$ induced by $h$ (which exists because $h$ is tangent to $\partial M$). We wish to show that either $X$ is not area-preserving, or the orbits of $X$ are purely toroidal.

Consider the space
\begin{equation*}
\cH_N^\Gamma(M) = \{H \in \Gamma(TM) : \curl H = 0, \, \Div H = 0, \, H \cdot n = 0\}
\end{equation*}
of vacuum fields on $M$. Let $\flat : \Gamma(TM) \to \Omega^1(M)$ denote the musical isomorphism from vector fields to 1-forms. Then, for $H \in \cH_N^\Gamma(M)$, $dH^\flat = 0$, and the map $\cH_N^\Gamma(M) \to H^1_{\text{dR}}(M)$ given by
\begin{equation*}
H \mapsto [H^\flat],
\end{equation*}
is an isomorphism (see \cite[Theorem 2.6.1 and Section 3.5]{Schwarz}). Because $\Psi_*$ is an isometry we have that
\begin{equation}\label{eq:pres-whole-space}
\Psi_*(\cH_N^\Gamma(M)) = \cH_N^\Gamma(M).     
\end{equation}
On the other hand, by Proposition \ref{prop:toroidal-index-on-solid-torus}, we have on $H_1(M)$ that
\begin{equation*}
\Psi_* = \sigma_2 \id,
\end{equation*}
where $\sigma_2$ is the toroidal index of $\cI$ on $\partial M$. Dually, on $H^1_{\text{dR}}(M)$, we have that
\begin{equation}\label{eq:action-on-cohomology-for-solid-torus}
\Psi^* = \sigma_2 \id.
\end{equation}
It follows from equations \eqref{eq:pres-whole-space} and \eqref{eq:action-on-cohomology-for-solid-torus} that
\begin{equation*}
\Psi^*(h^\flat) = \sigma_2 h.
\end{equation*}
Hence, because $\sigma_2^2 = 1$ and $\Psi$ is an isometry,
\begin{equation}\label{eq:action-of-psi-on-h}
\Psi_* h = \sigma_2 h.
\end{equation}

On the other hand, considering the diffeomorphism $\psi : \partial M \to \partial M$ induced from $\cI$, Equation \eqref{eq:action-of-psi-on-h} implies that
\begin{equation}\label{eq:action-of-psi-on-X}
\psi_* X = \sigma_2 X.    
\end{equation}
We note that because $h$ is non-trivial, $X$ is not identically zero. Indeed, this is because $h|_{\partial M} = 0$ implies $h = 0$, which is known in a very general setting \cite[Theorem 3.4.4]{Schwarz}. See also \cite[Theorem 2.2]{Gerner21}.

Now, suppose that $X$ is area-preserving. Then, we show that $X$ has no zeros. To see this, we recall an observation made in \cite{Perrella22}. Give $\partial M$ the oriented Riemannian structure inherited from $M \subset \RR^3$. Let $\mu_\partial$ be the Riemannian area-form on $\partial M$. Then, if $\mu$ is a positively oriented area-form which $X$ preserves, we can write
\begin{equation*}
\mu = P\mu_{\partial},    
\end{equation*}
for some positive $P \in C^{\infty}(\partial M)$ and
\begin{equation*}
\cL_{PX}\mu_\partial = d i_{PX}\mu_{\partial} = d i_X \mu = \cL_{X}\mu = 0.
\end{equation*}
Thus, if $\omega = X^{\flat}$ and $\delta : \Omega^1(\partial M) \to \Omega^0(\partial M)$ denotes the co-differential, then,
\begin{equation*}
d\omega = 0, \qquad \delta P \omega = 0,    
\end{equation*}
where $d\omega = 0$ is implied by $d h^\flat = 0$. That is, $\omega$ is what is known as a $P$-harmonic 1-form. As per \cite[Theorem 7]{Perrella22}, either $\omega = 0$ or $\omega$ has no zeros. Because $X$ is not identically zero, it follows that $X$ has no zeros. Hence, from Equation \eqref{eq:action-of-psi-on-X} and Proposition \ref{prop:periodic-orbits}, the orbits of $X$ are purely toroidal.
\end{proof} 

In order to prove Theorem \ref{thm:main}, we now move on to investigate how the poloidal and toroidal index behave in the setting of trivially foliated tori.

\section{Action on homology by a diffeomorphism in \texorpdfstring{$\RR^3$}{R3} on foliated tori}\label{sec:extension-to-nested-toroidal-surfaces}

Here we investigate homological results from Section \ref{sec:poloidal-and-toroidal-homologies} further, and in particular, prove Proposition \ref{prop:basic-rigidity}. We will use the following lemma.

\begin{restatable}{lemma}{diffactcom}\label{lem:diffeomorphism-action-compatibility} 
Let $\epsilon > 0$ and $J \in \{[0,\epsilon),(-\epsilon,\epsilon)\}$ be an interval. Let $S$ be a 2-torus and $\Sigma : S \times J \to \RR^3$ be an embedding. Let $S_j = \Sigma(S \times \{j\})$ for $j \in J$. Suppose that $\cI : \RR^3 \to \RR^3$ is a diffeomorphism such that $\cI(S_j) = S_j$ for all $j \in J$. Let $(\sigma_1)_j$ and $(\sigma_2)_j$ be the poloidal and toroidal indices of $\cI$ on $S_j$ for $j \in J$. Then, for $j \in J$, setting $\sigma_1 = (\sigma_1)_0$ and $\sigma_2 = (\sigma_2)_0$, 
\begin{equation*}
(\sigma_1)_j = \sigma_1, \qquad (\sigma_2)_j = \sigma_2.     
\end{equation*}
\end{restatable}

Lemma \ref{lem:diffeomorphism-action-compatibility} is proven in Appendix \ref{app:deferred-proofs-2}. We will apply Lemma \ref{lem:diffeomorphism-action-compatibility} to the setting of toroidally nested surfaces. We work with the following definition (see also \cite{Cardona23}).

\begin{definition}\label{def:toroial-lvl-sets}
Let $M \subset \RR^3$ be a solid toridal domain. That is, $M$ is a regular domain in $\RR^3$ diffeomorphic to $D \times \Sone$. Let $\gamma \subset \Int M$ be an embedded circle. A smooth function $p \in C^{\infty}(M)$ is said to have \emph{toroidally nested level sets with axis $\gamma$} if $\partial M$ is a regular level set of $p$ and each (non-empty) level set of $p$ is either
\begin{enumerate}
    \item a regular level set diffeomorphic to $\Sone\times\Sone$,
    \item or equal to $\gamma$.
\end{enumerate}
\end{definition}

For a smooth function $f : M \to \RR$ on a manifold $M$ with boundary, we let $\RegImag(f)$ denote the set of regular values of $f$ in its image $\Imag(f)$. Now for a lemma.

\begin{lemma}\label{lem:sigma-toroidal-level-sets}
Let $M \subset \RR^3$ be a solid toroidal domain. Let $\cI : \RR^3 \to \RR^3$ be a diffeomorphism such that $\cI(M) = M$. Then, $\cI(\partial M) = \partial M$. Let $\sigma_1,\sigma_2$ be the poloidal and toroidal index of $\cI$ on $\partial M$. Now, let $p \in C^{\infty}(M)$ be a smooth function on $M$ with toroidally nested level sets such that $p \circ \Psi = \Psi$, where $\Psi : M \to M$ is the diffeomorphism induced by $\cI$. For each $y \in \RegImag(p)$, $p^{-1}(y)$ is an embedded 2-torus such that $\cI(p^{-1}(y)) = p^{-1}(y)$ and if $(\sigma_1)_y,(\sigma_2)_y$ are the poloidal and toroidal index for $\cI$ on $S_y = p^{-1}(y)$, then,
\begin{equation*}
(\sigma_1)_y = \sigma_1, \qquad (\sigma_2)_y = \sigma_2.
\end{equation*}
\end{lemma}

\begin{proof}
As a preliminary, set $U = \{x \in M : \nabla p|_x \neq 0\}$ and consider the local vector field
\begin{equation*}
X = \frac{\nabla p}{\|\nabla p\|^2} \bigg|_U
\end{equation*}
on $U$. Let $y \in \RegImag(p)$. Then, $S_y = p^{-1}(y)$ is contained in $U$. Applying the Flowout Theorem \cite[Theorem 9.20]{Lee12} (in the case that $S_y \subset \Int M$) and applying the Boundary Flowout Theorem \cite[Theorem 9.24]{Lee12} (in the remaining case that $S_y = \partial M$), we find an $\epsilon > 0$, and an interval $J \in \{[0,\epsilon),(-\epsilon,\epsilon)\}$, and an open embedding $\Sigma : S_y \times J \to M$ such that, for $x \in S_y$,
\begin{equation*}
\Sigma(x,0) = x,
\end{equation*}
and such that, setting $c_x : J \to M$ to be the curve given by $c_x(t) = \Sigma(x,t)$,
\begin{equation*}
c_x'(t) = X(c_x(t)).
\end{equation*}
In particular, for $(x,t) \in S_y \times J$, it follows that
\begin{equation*}
p(\Sigma(x,t)) = y + t.     
\end{equation*}
For $j \in J$, $\Sigma(S_y \times j)$ is an embedded 2-torus on which $p$ is constant and regular. Because $p$ has toroidally nested level sets, it follows that
\begin{equation*}
p^{-1}(y + j) = \Sigma(S_y \times \{j\}).    
\end{equation*}
Hence, by Lemma \ref{lem:diffeomorphism-action-compatibility}, for all $j \in J$,
\begin{equation}\label{eq:continuity-condition}
(\sigma_1)_{y+j} = (\sigma_1)_y, \qquad (\sigma_2)_{y+j} = (\sigma_2)_y.
\end{equation}
Now, because $p$ has toroidally nested level sets, writing $\Imag(p) = [a,b]$ for some $a < b$, we have that $\RegImag(p) = \Imag(p) \setminus \{c\}$, where $p^{-1}(c) = \gamma$, the axis of $p$. Necessarily $c\in \{a,b\}$. Moreover, the value $d \in \{a,b\}$ with $d \neq c$ is the value of $p$ on $\partial M$. Therefore, from Equation \eqref{eq:continuity-condition}, the set
\begin{equation*}
V = \{y \in \RegImag(p) : (\sigma_1)_{y} = \sigma_1, \, (\sigma_2)_{y} = \sigma_2\},
\end{equation*}
is open and closed and contains $d$. Hence, $V = \RegImag(p)$ and the result follows.
\end{proof}

We now prove Proposition \ref{prop:basic-rigidity}.

\begin{proof}[Proof of Proposition \ref{prop:basic-rigidity}]
Let $M \subset \RR^3$ be a solid toroidal domain. Suppose that $B$ is a vector field on $M$ and $\rho$ is a smooth function on $M$ such that
\begin{equation*}
\Div B = 0, \qquad B \cdot \nabla \rho = 0 = \curl B \cdot \nabla \rho.
\end{equation*}
Suppose further that $\rho$ has toroidally nested level sets. In the following, we make use of the theory of $P$-harmonic forms studied in \cite{Perrella22}. Specifically, for $y \in \RegImag(\rho)$, $S_y = \rho^{-1}(y)$ is an embedded 2-torus and letting $\iota_y : S_y \subset M$ denote the inclusion, the 1-form $b_y = \iota_y^*B^{\flat}$, satisfies
\begin{equation}\label{eq:P-harmonic}
db_y = 0, \qquad \delta P_y b_y = 0, 
\end{equation}
where $\delta : \Omega^1(S_y) \to \Omega^0(S_y)$ is the co-differential on $S_y$ inherited from the metric on $M$ and the orientation induced by $\nabla p$, and
\begin{equation*}
P_y = \frac{1}{\|\nabla \rho\||_{S_y}}.    
\end{equation*}
Solutions to Equation \eqref{eq:P-harmonic} are called $P$-harmonic 1-forms. In particular, for each cohmology class $C \in H^1_{\text{dR}}(S_y)$, there exists a unique 1-form $\kappa_C \in \Omega^1(S_y)$ such that
\begin{equation}\label{eq:P-harmonic-perscribed-class}
d\kappa_C = 0, \qquad \delta P_y \kappa_C = 0, \qquad [\kappa_C] = C. 
\end{equation}

Let $\cI : \RR^3 \to \RR^3$ be an isometry such that $\cI(M) = M$ and $p \circ \Psi = p$. Because $\Psi$ is an isometry, we have that
\begin{equation*}
\|\nabla p\| \circ \Psi = \|\nabla p\|.
\end{equation*}
For any $y \in \RegImag(\rho)$, considering the isometry $\psi_y : S_y \to S_y$ induced by $\Psi : M \to M$, we therefore have that
\begin{equation*}
P_y \circ \psi_y = \psi_y.    
\end{equation*}
Thus, if $C \in H^1_{\text{dR}}(S_y)$ is a given cohomology class, and $\kappa_C$ is the unique solution to Equation \eqref{eq:P-harmonic-perscribed-class}, then $\psi_y^*\kappa_C$ satisfies
\begin{equation*} 
d\psi_y^*\kappa_C = 0, \qquad \delta P_y \psi_y^*\kappa_C = 0, \qquad [\psi_y^*\kappa_C] = \psi_y^*C. 
\end{equation*}
Thus, adopting the notation from Equation \eqref{eq:P-harmonic-perscribed-class}, we have that, for $C \in H^1_{\text{dR}}(S_y)$,
\begin{equation}\label{eq:summary-of-P-harmonic}
\psi_y^*(\kappa_C) = \kappa_{\psi_y^*C}.
\end{equation}

Now, assume that $\cI$ is orientation-preserving. Let $\sigma_1,\sigma_2$ be the poloidal and toroidal index of $\cI$ on $\partial M$. Then, we have that $\sigma_1 = \sigma_2$. Denote this common number as $\sigma$. Now, for $y \in \RegImag(\rho)$, setting $S_y = \rho^{-1}(y)$, we have that $\cI(S_y) = S_y$, and letting $(\sigma_1)_y,(\sigma_2)_y$ be the poloidal and toroidal index for $\cI$ on $S_y$, from Lemma \ref{lem:sigma-toroidal-level-sets}, we have that,
\begin{equation*}
(\sigma_1)_y = \sigma, \qquad (\sigma_2)_y = \sigma.
\end{equation*}
Considering the induced diffeomorphism $\psi_y : S_y \to S_y$ from $\cI$, we therefore have that the induced map $(\psi_y)_* : H_1(S_y) \to H_1(S_y)$ satisfies
\begin{equation*}
(\psi_y)_* = \sigma \id.    
\end{equation*}
However this means that the induced map $(\psi_y)^* : H^1_{\text{dR}}(S_y) \to H^1_{\text{dR}}(S_y)$ satisfies
\begin{equation*}
(\psi_y)^* = \sigma \id.    
\end{equation*}
Therefore, Equation \eqref{eq:summary-of-P-harmonic} reads
\begin{equation}\label{eq:summary-of-P-harmonic-for-orientation-pres}
\psi_y^*(\kappa_C) = \sigma \kappa_{C}.
\end{equation}
Hence, for each $y \in \RegImag(\rho)$, considering $b_y$ in Equation \eqref{eq:P-harmonic}, we have that
\begin{equation*}
\psi_y^* b_y = \sigma b_y.    
\end{equation*}
However, because $\psi_y$ is an isometry, $\sigma^2 = 1$, and $B$ is tangent to $S_y$, it follows that
\begin{equation*}
(\Psi_* B)|_{S_y} = \sigma B|_{S_y}.    
\end{equation*}
But, because $\rho^{-1}(\RegImag(\rho)) = M \setminus \gamma$ is dense in $M$ (as $\gamma$ is an embedded 1-dimensional manifold, having measure zero in $M$), it follows that
\begin{equation*}
\Psi_* B = \sigma B.
\end{equation*}

For the orientation-reversing case, observe that $\cJ = \cI \circ \cI$ is orientation-preserving. Letting $\sigma_1,\sigma_2$ be the poloidal and toroidal index of $\cI$ on $\partial M$, the poloidal and toroidal indices of $\cJ$ are $\sigma_1^2 = 1$ and $\sigma_2^2 = 1$. Thus, the previous paragraph may be applied to $\cJ$ with $\sigma = 1$. That is, one obtains that
\begin{equation*}
(\Psi \circ \Psi)_* B = B.
\end{equation*}
\end{proof}

\section{Proof of the main result}\label{sec:proof-main-res}

The proof of Theorem \ref{thm:main} makes use of the following observation which is also used in other parts of the paper.

\begin{remark}\label{rem:B-pres-implies-p-and-J-pres}
Let $M$ be a compact, connected, and oriented Riemannian 3-manifold with boundary. Suppose that $\Psi : M \to M$ is an isometry and $(B,p)$ is an MHD equilibrium such that
\begin{equation}\label{eq:psi-action-on-B}
\Psi_*B = \sigma B,
\end{equation}
for some $\sigma \in \{+1,-1\}$. We will show that this implies $\Psi^*p = p$. Indeed, letting $g$ be the metric on $M$ and $b = i_B g$, it is well-known (see \cite{Peralta-Salas21}) that the equation $\curl B \times B = \nabla p$ can be equivalently written as
\begin{equation*}
i_Bdb = dp.
\end{equation*}
Equation \eqref{eq:psi-action-on-B} and $\sigma^2 = 1$ imply that $\Psi^*B = (\Psi^{-1})_* B = \sigma B$. Hence, $\Psi^*b = \sigma b$ because $\Psi^*g = g$. Thus,
\begin{equation*}
\Psi^*dp = \Psi^*(i_Bdb) = i_{\Psi^*B}(d\Psi^*b) = \sigma^2i_Bdb = dp.
\end{equation*}
Thus, because $M$ is connected, there exists a constant $c$ so that
\begin{equation*}
\Psi^*p = p + c.
\end{equation*}
Because $M$ is compact and connected, $\Imag(p)$ is a closed and bounded interval in $\RR$. Because $\Psi$ is a surjective, $\Imag(\Psi^*p) = \Imag(p)$. Thus, $c = 0$ and so, as claimed, 
\begin{equation*}
\Psi^*p = p.
\end{equation*}
Of course, $J = \curl B$ also satisfies a relation. If $\mu$ denotes the Riemanian volume-form on $(M,g)$, then, because $M$ is connected,
\begin{equation*}
\Psi^*\mu = \kappa \mu,
\end{equation*}
with $\kappa = 1$ if $\Psi$ is orientation-preserving and $\kappa = -1$ if $\Psi$ is orientation-reversing. The vector field $J$ satisfies
\begin{equation*}
i_J\mu = db.    
\end{equation*}
Hence, on the one hand,
\begin{equation*}
\Psi^*(i_J\mu) = i_{\Psi^*J}\Psi^*\mu = i_{(k\Psi^*J)}\mu     
\end{equation*}
but on the other hand,
\begin{equation*}
\Psi^*(i_J\mu) = \Psi^*(db) = d\Psi^*b = \sigma db = i_{\sigma J}\mu. 
\end{equation*}
Hence (because $k^2 = 1$),
\begin{equation*}
\Psi_* J = k \sigma J.    
\end{equation*}
\end{remark}

We now prove Theorem \ref{thm:main}.

\begin{proof}[Proof of Theorem \ref{thm:main}]
We recall the setting. Let $\cI : \RR^3 \to \RR^3$ be an orientation-reversing isometry. Let $M \subset \RR^3$ be a solid toroidal domain. Suppose that $\cI(M) = M$. Then $\cI(\partial M) = \partial M$. Let $\sigma_1$ and $\sigma_2$ denote respectively the poloidal and toroidal index of $\cI$ on the 2-torus $\partial M$. Let $(B,p)$ be an MHD equilibrium on $M$ such that $\Psi$ and $B$ satisfy Equation \eqref{eq:psi-action-on-B} for some $\sigma \in \{+1,-1\}$, where $\Psi : M \to M$ is the diffeomorphism on $M$ induced by $\cI$. Suppose further that $p$ has toroidally nested level sets. Let $y \in \RegImag(p)$ be a regular value in the image of $p$. Then, $B$ and $J = \curl B$ are tangent to the the 2-torus $S_y = p^{-1}(y)$. Let $B_y$ and $J_y$ be the vector fields respectively induced by $B$ and $J$. We wish to prove the following.
\begin{enumerate}
    \item If $\sigma = \sigma_1$, then the orbits of $B_y$ are purely poloidal and the orbits of $J_y$ are purely toroidal.
    \item If $\sigma = \sigma_2$, then the orbits of $B_y$ are purely toroidal and the orbits of $J_y$ are purely poloidal.
\end{enumerate}

From Remark \ref{rem:B-pres-implies-p-and-J-pres}, we have that
\begin{equation*}
\Psi^*p = p
\end{equation*}
and so, by Lemma \ref{lem:sigma-toroidal-level-sets}, the 2-torus $S_y$ satisfies
\begin{equation*}
\cI(S_y) = S_y,    
\end{equation*}
and the poloidal and toroidal index of $\cI$ on $S_y$ are, respectively, $\sigma_1$ and $\sigma_2$.

A well-known consequence \cite{Burby20,Perrella22,Perrella23} of $\Div B = 0 = \Div J$ and $B \cdot \nabla p = 0 = J \cdot \nabla p$ is that the induced vector fields $J_y$ and $B_y$ on $S_y$ preserve the area-form
\begin{equation*}
\tilde{\mu}_y = P_y \mu_y,
\end{equation*}
where $\mu_y$ is the area-form on $S_y = p^{-1}(y)$ inherited from the Euclidean metric on $M \subset \RR^3$, and
\begin{equation*}
P_y = \frac{1}{\|\nabla p\||_{S_y}}.    
\end{equation*}
If $\psi_y : S_y \to S_y$ denotes the diffeomorphism on $S_y$ induced by $\cI$, the relation $\Psi_*B = \sigma B$ implies that
\begin{equation*}
(\psi_y)_*B_y = \sigma B_y.
\end{equation*}
From Remark \ref{rem:B-pres-implies-p-and-J-pres} we also know that $\Psi_* J = -\sigma J$, which implies that
\begin{equation*}
(\psi_y)_*J_y = -\sigma J_y.
\end{equation*}
However, both $B_y$ and $J_y$ are non-vanishing on $S_y$ because $\nabla p|_{S_y}$ is non-vanishing. Hence, the result follows from Proposition \ref{prop:periodic-orbits}.
\end{proof}

\section{Example computation}\label{sec:example-comp}

We will now consider some examples in which it is simple to compute the toroidal index. As a consequence, we will prove Corollary \ref{cor:Stellarator-symmetry}. To start, let $\cI : \RR^3 \to \RR^3$ be any diffeomorphism such that
\begin{equation*}
\cI(\cZ) = \cZ,    
\end{equation*}
where $\cZ = \myspan\{(0,0,1)\}$ is the $z$-axis. Then, consider the curve $\gamma_0 : [0,1] \to \RR^3 \setminus \cZ$ given by
\begin{equation}\label{eq:gamma0-def}
\gamma_0(t) = (\cos(2\pi t), \sin(2\pi t),0).     
\end{equation}
Because $H_1(\RR^3 \setminus \cZ) \cong \ZZ$ and $\cI$ descends to a diffeomorphism $\cJ : \RR^3\setminus \cZ \to \RR^3 \setminus \cZ$, there exists (a unique) $\sigma \in \{-1,1\}$ such that, on $H_1(\RR^3 \setminus \cZ)$,
\begin{equation}\label{eq:homoloogy-not-on-Z}
\cJ_* = \sigma \id.
\end{equation}
The $\sigma$ in Equation \eqref{eq:homoloogy-not-on-Z} can be easily be deduced from the fact that
\begin{equation}\label{eq:action-on-gamma0}
[\cJ \circ \gamma_0] = \sigma[\gamma_0].
\end{equation}
We will discuss a situation in which the $\sigma$ solving equations \eqref{eq:homoloogy-not-on-Z} and \eqref{eq:action-on-gamma0} is also the toroidal index of $\cI$ on the boundary of a solid torus.

Let $M$ be a solid torus in $\RR^3$ which wraps around $\cZ$ once, in the sense that
\begin{enumerate}
    \item $M \cap \cZ = \emptyset$ and
    \item letting $\iota : M \subset \RR^3 \setminus \cZ$ denote the inclusion, $\iota_* : H_1(M) \to H_1(\RR^3 \setminus \cZ)$ is surjective.
\end{enumerate}
Let $\Psi : M \to M$ be the diffeomorphism induced by $\cI$ on $M$. Then, we have that
\begin{equation*}
\iota \circ \Psi = \cJ \circ \iota.
\end{equation*}
Therefore, on first homology, denoting by $\id$ the identities on $H_1(\RR^3 \setminus \cZ)$ and $H_1(M)$,
\begin{align*}
\iota_* \circ \Psi_* &= \cJ_* \circ \iota_*\\
&= (\sigma \id) \circ \iota_*\\
&= \iota_* \circ (\sigma \id).
\end{align*}
However, because $H_1(M) \cong \ZZ$ and $\iota_* : H_1(M) \to H_1(\RR^3 \setminus \cZ)$ is a surjective homomorphism between groups isomorphic to $\ZZ$, and so $\iota_*$ is also injective. Hence,
\begin{equation}\label{eq:Psi-action-on-first-hom}
\Psi_* = \sigma \id.
\end{equation}

On the other hand, $\cI(\partial M) = \partial M$ and letting $\sigma_2$ denote the toroidal index of $\cI$ on $\partial M$, by Proposition \ref{prop:toroidal-index-on-solid-torus},
\begin{equation*}
\Psi_* = \sigma_2 \id.
\end{equation*}
Therefore,
\begin{equation}\label{eq:tor-index-is-sigma}
\sigma_2 = \sigma. 
\end{equation}
We now use this computation to prove Corollary \ref{cor:Stellarator-symmetry}.

\begin{proof}[Proof of Corollary \ref{cor:Stellarator-symmetry}]
Let $\cI : \RR^3 \to \RR^3$ be the orientation-preserving isometry given by
\begin{equation}
\cI(x,y,z) = (x,-y,-z)  
\end{equation}
and suppose that $\cI(M) = M$. Let $\Psi : M \to M$ denote the corresponding diffeomorphism. We wish to prove that, for any MHD equilibrium $(B,p)$ on $M$ with $p$ having toroidally nested level sets, one has the equivalence
\begin{equation*}
p \circ \Psi = p \iff \Psi_* B = - B.  
\end{equation*}

To find the corresponding $\sigma$ solving equations \eqref{eq:homoloogy-not-on-Z} and \eqref{eq:action-on-gamma0}, we compute for $t \in [0,1]$ that
\begin{equation*}
\cJ(\gamma_0(t)) = \cI(\gamma_0(t)) = (\cos(2\pi t), -\sin(2\pi t),0) = \gamma_0(1-t).    
\end{equation*}
Hence,
\begin{equation*}
[\cJ \circ \gamma_0] = -[\gamma_0].
\end{equation*}
That is, $\sigma = -1$. Hence, by Equation \eqref{eq:tor-index-is-sigma}, the toroidal index $\sigma_2$ of $\cI$ on $\partial M$ is $\sigma_2 = \sigma = -1$. Because $\cI$ is orientation-preserving, it follows from Proposition \ref{prop:homology-action-characterisation} that the poloidal index $\sigma_1$ of $\cI$ on $\partial M$ is also $-1$.

Now, let $(B,p)$ be an MHD equilibrium on $M$ where $p$ has toroidally nested level sets. Suppose that $p \circ \Psi = p$. Then, from Proposition \ref{prop:basic-rigidity}, we immediately obtain that
\begin{equation*}
\Psi_*B = \sigma B = -B.    
\end{equation*}
For the other direction, suppose that $\Psi_* B = -B$. Then, immediately from Remark \ref{rem:B-pres-implies-p-and-J-pres},
\begin{equation*}
p \circ \Psi = p.    
\end{equation*}
\end{proof}

\section{Some remarks on magnetic axes}\label{sec:magnetic axes}

In this section, we prove Proposition \ref{prop:nice-axis-behaviour}. To start, let $(B,p)$ be an MHD equilibrium on a Riemannian manifold $M$ with boundary. Then, setting
\begin{equation*}
J = \curl B,    
\end{equation*}
one has that
\begin{equation}\label{eq:laplacian-of-p}
\begin{split}
\Delta p &= \Div(\nabla p)\\
&= \Div(\curl B \times B)\\
&= \curl(\curl B) \cdot B - \curl B \cdot \curl B\\
&= \curl J \cdot B - J \cdot J.
\end{split}
\end{equation}
Equation \eqref{eq:laplacian-of-p} is helpful for gaining information about zeros of $B$ on Morse-Bott critical circles of $p$.

\begin{lemma}\label{lem:implication-of-zeros}
Let $\gamma \subset \Int M$ be an embedded circle which is a Morse-Bott critical submanifold for $p$. The latter meaning that, $d p|_\gamma = 0$, and, for all $x \in \gamma$ the Hessian $\hess(p)|_x : T_x M \times T_x M \to \RR$ is non-degenerate once restricted to one (and hence any) plane $\pi \subset T_x M$ which is transverse to $\gamma$ at $x$. Then, the following holds.
\begin{enumerate}
    \item If $p$ is locally maximised on $\gamma$, then, for $x \in \gamma$, $J|_x = 0 \implies B_x \neq 0$.
    \item If $p$ is locally minimised on $\gamma$, then $B|_\gamma$ has no zeros.
\end{enumerate} 
\end{lemma} 

\begin{proof}
Recall that for $x \in M$, the trace of $\hess(p)|_x$ is $\Delta p|_x$. Suppose that $p$ is locally maximised on $\gamma$. Then, for $x \in \gamma$, from the non-degeneracy assumption placed on $\hess(p)|_x$
\begin{equation*}
\Delta p|_x = \tr{\hess(p)|_x} < 0,
\end{equation*}
and so, by Equation \eqref{eq:laplacian-of-p},
\begin{equation*}
\curl J \cdot B|_x < J \cdot J |_x.    
\end{equation*}
Therefore, we obtain the implication
\begin{equation*}
J|_x = 0 \implies B|_x \neq 0.    
\end{equation*}

Suppose that $p$ is locally minimised on $\gamma$. Then, for $x \in \gamma$,
\begin{equation*}
\Delta p|_x = \tr{\hess(p)|_x} > 0,
\end{equation*}
and so, by Equation \eqref{eq:laplacian-of-p},
\begin{equation*}
\curl J \cdot B|_x > J \cdot J |_x.    
\end{equation*}
Moreover, because $J \cdot J \geq 0$, it follows that
\begin{equation*}
\curl J \cdot B|_x > 0.
\end{equation*}
Therefore, $B|_x \neq 0$. That is, $B|_\gamma$ has no zeros.
\end{proof}

We can strengthen the result of Lemma \ref{lem:implication-of-zeros} (to prove Proposition \ref{prop:nice-axis-behaviour}) through the well-known fact that $J$ and $B$ commute; namely
\begin{equation}\label{eq:J-B-commute}
[J,B] = 0.
\end{equation}
The following lemma prepares use of Equation \eqref{eq:J-B-commute}

\begin{lemma}\label{lem:decend}
Let $\gamma \subset \Int M$ be a Morse-Bott critical submanifold of $p$. Then, $J$ and $B$ are tangent to $\gamma$, and so there exist unique smooth vector fields $J_\gamma$ and $B_\gamma$ on $\gamma$ such that
\begin{equation*}
T\ell \circ J_\gamma = J \circ \ell, \qquad T\ell \circ B_\gamma = B \circ \ell, \qquad [J_\gamma,B_\gamma] = 0,
\end{equation*}
where $\ell : \gamma \subset M$ denotes the inclusion.
\end{lemma}

\begin{proof}
Let $X \in \{B,J\}$. Then, we have that $dp(X) = 0$. Hence, for any vector field $Y$, and for any $x \in \gamma$,
\begin{equation*}
\hess(p)|_x(X|_x,Y|_x) = \hess(p)|_x(Y|_x,X|_x) = Y(dp(X))|_x = 0.   
\end{equation*}
Thus, $X|_x$ is in the kernel of $\hess(p)|_x$. Hence, by the definition of Morse-Bott critical submanifold, we must have that $X|_x$ is tangent to $\gamma$ at $x$. Hence, $J$ and $B$ are tangent to $\gamma$. Because $\gamma$ is a submanifold, there exist unique vector fields $J_\gamma$ and $B_\gamma$ on $\gamma$ such that
\begin{equation*}
T\ell \circ J_\gamma = J \circ \ell, \qquad T\ell \circ B_\gamma = B \circ \ell,
\end{equation*}
where $\ell : \gamma \subset M$ denotes the inclusion. Then, from Equation \eqref{eq:J-B-commute}, one finds that
\begin{equation*}
T\ell \circ [J_\gamma,B_\gamma] = [J,B] \circ \ell = 0,
\end{equation*}
and so $[J_\gamma,B_\gamma] = 0$.
\end{proof}

We now state a final lemma to prove Proposition \ref{prop:nice-axis-behaviour}.

\begin{lemma}\label{lem:commuting-1dim-fields}
Let $X,Y$ be smooth vector fields on a connected 1-manifold $L$ such that
\begin{equation*}
[X,Y] = 0,    
\end{equation*}
and such that, for all $x \in L$,
\begin{equation}\label{eq:implication}
X|_x = 0 \implies Y|_x \neq 0.    
\end{equation}
Then, either
\begin{enumerate}
    \item $X = 0$ and $Y$ has no zeros, or
    \item $X$ has no zeros and $Y = k X$ for some $k \in \RR$.
\end{enumerate}
\end{lemma} 

\begin{proof}
Consider the set $\cZ = \{x \in L : X|_x = 0\}$. Then $\cZ$ is a closed subset of $L$. On the other hand, the relation $[X,Y] = 0$ implies that $X$ is invariant under the (local) flow of $Y$, and so $\cZ$ is closed under the flow of $Y$. That is, for each $x \in \cZ$, we can find $\epsilon > 0$ such that the maximally defined time-$t$ flow $\psi^Y_t : L_t \to L_{-t}$ contains $x$ in its domain for $|t| < \epsilon$ and $\psi^Y_t(x) \in \cZ$. However, $Y|_x \neq 0$ from Equation \eqref{eq:implication} and so, because $\dim L = 1$, the Inverse Function Theorem implies that $\psi^Y((-\epsilon,\epsilon) \times \{x\}) \subset \cZ$ is an open neighborhood of $x$ in $L$. Hence, $\cZ$ is also an open subset of $L$. Because $L$ is connected, it follows that $\cZ = L$ or $\cZ = \emptyset$.

That is, $X = 0$ or $X$ has no zeros. If $X = 0$, then Equation \eqref{eq:implication} implies that $Y$ has no zeros. If $X$ has no zeros, then $Y = fX$ for some smooth $f \in C^{\infty}(L)$ (because $\dim L = 1$). Moreover,
\begin{equation*}
X(f)X = [X,fX] = [X,Y] = 0,    
\end{equation*}
and so $X(f) = 0$. Thus, because $X$ has no zeros, $df = 0$ (as $\dim L = 1$) and so $f$ is constant because $L$ is connected. That is, $Y = kX$ for some $k \in \RR$.
\end{proof}

We now prove Proposition \ref{prop:nice-axis-behaviour}.

\begin{proof}[Proof of Proposition \ref{prop:nice-axis-behaviour}]
Let $\gamma \subset \Int M$ be an embedded circle which is a Morse-Bott critical submanifold for $p$. Suppose that $p$ is locally maximised on $\gamma$. Then, from Lemma \ref{lem:implication-of-zeros}, for $x \in \gamma$, $J|_x = 0$ implies $B_x \neq 0$. Inheriting the notation from Lemma \ref{lem:decend}, we therefore have that, for $x \in \gamma$,
\begin{equation*}
{J_\gamma}|_x = 0 \implies {B_\gamma}|_x \neq 0.
\end{equation*}
Taking the contra-positive, for $x \in \gamma$,
\begin{equation*}
{B_\gamma}|_x = 0 \implies {J_\gamma}|_x \neq 0.
\end{equation*}
Hence, from Lemma \ref{lem:commuting-1dim-fields}, either
\begin{enumerate}
    \item $B_\gamma = 0$ and $J_\gamma$ has no zeros, or
    \item $B_\gamma$ has no zeros and $J_\gamma = k B_\gamma$ for some $k \in \RR$.
\end{enumerate}

Suppose $p$ is locally minimised on $\gamma$. Then, immediately from Lemma \ref{lem:implication-of-zeros}, $B|_\gamma$ has no zeros. Therefore, one vacuously has the implication ${B_\gamma}|_x = 0 \implies {J_\gamma}|_x \neq 0$ for $x \in \gamma$. Hence, from Lemma \ref{lem:commuting-1dim-fields}, $J_\gamma = k B_\gamma$ for some $k \in \RR$.
\end{proof}

\section{Action by a diffeomorphism on the first cohomology of surfaces}\label{sec:sympectic-algebra}

Here we will consider some algebra related to how some diffeomorphisms act on the first cohomology of closed oriented connected 2-manifolds. In Section \ref{sec:sym-lin-alg}, we will recall some basic linear algebra, and in Section \ref{sec:app-to-surf}, we will apply this to surfaces and prove an extension of Proposition \ref{prop:periodic-orbits}.

\subsection{Some linear algebra}\label{sec:sym-lin-alg}

Let $V$ be a finite dimensional vector space and let $L : V \to V$ be a linear map such that
\begin{equation*}
L^2 = \id.    
\end{equation*}
For $\sigma \in \{+1,-1\}$, we set
\begin{equation*}
E_{\sigma} = \{v \in V : Lv = \sigma v\}.
\end{equation*}
A well-known result is the following.

\begin{lemma}
The vector space $V$ admits the internal direct sum
\begin{equation*}
V = E_{+1} \oplus E_{-1}.    
\end{equation*}
\end{lemma}

\begin{proof}
Of course, we have 
\begin{equation}\label{eq:trivial-intersection}
E_{+1} \cap E_{-1} = \{0\}.
\end{equation}
Moreover, for each $v \in V$, we have
\begin{equation*}
v = v_1 + v_{-1},    
\end{equation*}
where
\begin{equation*}
v_1 = \frac{v+Lv}{2}, \qquad v_{-1} = \frac{v-Lv}{2},
\end{equation*}
and
\begin{equation*}
Lv_1 = v_1, \qquad Lv_{-1} = -v_{-1}.   
\end{equation*}
Hence,
\begin{equation}\label{eq:sum}
V = E_{+1} + E_{-1}.    
\end{equation}
The result now follows from equations \eqref{eq:trivial-intersection} and \eqref{eq:sum}.
\end{proof}

\begin{lemma}
Suppose that $\dim V = N$ and $\Lambda$ is a lattice in $V$ of rank $N$. That is, $\Lambda$ is a subset of $V$ for which there exist linearly independent $e_1,...,e_N \in V$ such that
\begin{equation*}
\Lambda = \mathbb{Z}e_1 \oplus ... \oplus \mathbb{Z}e_N. 
\end{equation*}
Suppose that $L(\Lambda) \subset \Lambda$. Then, for $\sigma \in \{+1,-1\}$,
\begin{equation*}
\myspan{(E_{\sigma} \cap \Lambda)} = E_{\sigma}.
\end{equation*}
\end{lemma}

\begin{proof}
Taking $e_1,..,e_N$ as above, $e_1,...,e_N$ forms a basis for $V$. Hence, for $v \in E_{\sigma}$, there exist unique $a_1,...,a_N \in \RR$, that
\begin{equation*}
v = a_1 e_1 + ... + a_N e_N.   
\end{equation*}
For each $i \in \{1,...,N\}$ and $s \in \{+1,1\}$, we set
\begin{equation*}
e_{i,s} = e_i + s Le_i.    
\end{equation*}
Then $e_{i,s} \in E_s$ for all $i \in \{1,...,N\}$ and $s \in \{+1,1\}$. Moreover,
\begin{align*}
v &= \frac{a_1}{2}e_{1,{\sigma}} + \frac{a_1}{2}e_{1,{-\sigma}} + ... + \frac{a_N}{2}e_{N,{\sigma}} + \frac{a_N}{2}e_{N,{-\sigma}}\\
&= \left(\frac{a_1}{2}e_{1,{\sigma}} + ... + \frac{a_N}{2}e_{N,{\sigma}}\right) + \left(\frac{a_1}{2}e_{1,{-\sigma}} + ... + \frac{a_N}{2}e_{N,{-\sigma}}\right).
\end{align*}
We recall that $v \in E_\sigma$ and the first term in the last equality is in $\myspan(E_\sigma \cap \Lambda) \subset E_\sigma$ while the second term is in $\myspan(E_{-\sigma} \cap \Lambda) \subset E_{-\sigma}$. Thus, because $E_\sigma \cap E_{-\sigma} = \{0\}$, the second term vanishes and so
\begin{equation*}
v \in \myspan{(E_{\sigma} \cap \Lambda)}. 
\end{equation*}
This proves $E_{\sigma} \subset \myspan{(E_{\sigma} \cap \Lambda)}$. The other containment is clear, and so we are done.
\end{proof}

We have a final lemma which considers a symplectic form on $V$ which is compatible with $L$. 

\begin{lemma}
Let $\omega \in \Lambda^2(V)$ be a symplectic form on $V$. Suppose that
\begin{equation*}
L^*\omega = -\omega.    
\end{equation*}
Let $\sigma \in \{+1,-1\}$. Then $E_{\sigma}$ is a Lagrangian subspace. That is,
\begin{equation*}
E_{\sigma}^\perp = E_{\sigma},    
\end{equation*}
where $\perp$ denotes the symplectic complement. In particular,
\begin{equation*}
\dim E_{+1} = \frac{\dim V}{2} = \dim E_{-1}.
\end{equation*}
\end{lemma}

\begin{proof}
Let $u \in E_{\sigma}$. Let $v \in E_{\sigma}$. Then, we have
\begin{equation*}
\omega(u,v) = -\omega(Lu,Lv) = -\omega(\sigma u,\sigma v) = -\omega(u,v),    
\end{equation*}
and so $\omega(u,v) = 0$. Hence, $u \in E_{\sigma}^\perp$. That is,
\begin{equation*}
E_{\sigma} \subset E_{\sigma}^\perp.
\end{equation*}
Now let $u \in E_{\sigma}^\perp$. Then, let $v \in V$. We have that $\sigma v + Lv \in E_{\sigma}$. Thus,
\begin{align*}
0 &= \omega(u,\sigma v + Lv)\\
&= \omega(u,\sigma v) + \omega(u, Lv)\\
&= \omega(\sigma u, v) - \omega(Lu, LLv)\\
&= \omega(\sigma u, v) - \omega(Lu, v)\\
&= \omega(\sigma u - Lu, v).
\end{align*}
Now, because $\omega$ is symplectic, and $v \in V$ was arbitrary, $\sigma u - Lu = 0$ and so $u \in E_{\sigma}$. Thus,
\begin{equation*}
E_{\sigma}^\perp \subset E_{\sigma}.
\end{equation*}
The conclusion follows.
\end{proof}

\subsection{Application to surfaces}\label{sec:app-to-surf}

Let $S$ be a closed, connected, and oriented 2-manifold. We say that a cohomology class $C \in H^1_{\text{dR}}(S)$ has \textit{integral periods} if, for all $\gamma \in H_1(S)$,
\begin{equation*}
 \int_{\gamma}C \in \ZZ,   
\end{equation*}
where the integral represents the de Rham pairing between $H^1_{\text{dR}}(S)$ and $H_1(S)$. By de Rham's Theorem, the set
\begin{equation*}
H^1_{\text{dR},\ZZ}(S) = \{C \in H^1_{\text{dR}}(S) : C \text{ has integral periods}\},   
\end{equation*}
is a lattice of rank $b$, where $b = \dim H^1_{\text{dR}}(S)$.

We now explain the relevance between the lattice $H^1_{\text{dR},\ZZ}(S)$ and periodic orbits. To begin, for a smooth function $\phi : S \to \Sone$, we set
\begin{equation*}
d\phi = \phi^*dx,    
\end{equation*}
where $dx$ is the standard 1-form on $\Sone = \RR/\ZZ$. Then, one has the following well-known lemma \cite{Tischler70,Farber04}.

\begin{lemma}\label{lem:circle-valued-characterisation}
Let $S$ be a manifold and $\omega \in \Omega^1(S)$ be a closed 1-form. Then, $[\omega]$ has integral periods if and only if $\omega = d\phi$ for some smooth $\phi : S \to \Sone$.
\end{lemma}

This lemma has a useful consequence in proving existence of periodic orbits for area-preserving flows on compact surfaces. A similar observation was made in \cite{PerrellaThesis}.

\begin{lemma}\label{lem:periodic-orbits-for-abstract}
Let $\mu \in \Omega^2(S)$ be an area-form and $X$ a vector field on $S$ such that
\begin{equation*}
\cL_X \mu = 0.    
\end{equation*}
Let $\omega = i_X \mu$. Then, $d\omega = 0$. Suppose that there exists $a \in \RR \setminus \{0\}$ such that $[a\omega]$ has integral periods. Then, if $X \neq 0$, $X$ has a non-trivial periodic orbit.
\end{lemma}

\begin{proof}
As well-known,
\begin{equation*}
d\omega = di_X\mu = i_X d\mu + di_X\mu = \cL_X \mu = 0.
\end{equation*}
Now, suppose that there exists $a \neq 0$ such that $[a\omega]$ has integral periods. Then, by Lemma \ref{lem:circle-valued-characterisation}, $a\omega = d\phi$ for some smooth $\phi : S \to \Sone$. We may write that
\begin{equation}\label{eq:phi-and-X}
d\phi = ai_X\mu.    
\end{equation}

Suppose that $X \neq 0$. Then, $\omega = i_X \mu \neq 0$ and so $d\phi \neq 0$. Hence, the image $\Imag \phi$ has non-empty interior in $\Sone$. Therefore, by Sard's Theorem, there exists a regular value $y \in \Sone$ in the image of $\phi$. Hence, $\phi^{-1}(y)$ is an embedded submanifold of $S$ of dimension $1$. Moreover, $\phi^{-1}(y)$ is compact because it is a (topologically) closed subset of $S$. Let $\ell$ be a connected component of $\phi^{-1}(y)$. Then, $\ell$ is a embedded 1-manifold which is also compact and connected. Therefore, $\ell$ is an embedded circle in $S$. Because $d\phi|_{\phi^{-1}(y)}$ has no zeros, we have by Equation \eqref{eq:phi-and-X} that $X|_{\ell}$ has no zeros and is tangent to $\ell$. A standard connectedness argument then shows that $\ell$ is an orbit of $X$.
\end{proof}

On the other hand, by Poincar\'e duality, we have the symplectic form $W : H^1_{\text{dR}}(S) \times H^1_{\text{dR}}(S) \to \RR$ given by
\begin{equation*}
W([\omega],[\eta]) = \int_S \omega \wedge \eta,
\end{equation*}
where $\omega,\eta \in \Omega^1(S)$ are closed 1-forms.

For any orientation-reversing diffeomorphism $\psi : S \to S$, considering the linear map $L = \psi^* : H^1_{\text{dR}}(S) \to H^1_{\text{dR}}(S)$, one has that
\begin{equation*}
L^*W = -W.    
\end{equation*}
Indeed, for any closed 1-forms $\omega,\eta \in \Omega^1(S)$, because $\phi$ is orientation-reversing,
\begin{align*}
W(\psi^*[\omega],\psi^*[\eta]) &= W([\psi^*\omega],[\psi^*\eta])\\
&= \int_S \psi^*\omega \wedge \psi^*\eta\\
&= \int_S \psi^*(\omega \wedge \eta)\\
&= -\int_S \omega \wedge \eta\\
&= -W([\omega],[\eta]).
\end{align*}
We will henceforth assume in addition to $\psi$ being orientation-reversing that
\begin{equation*}
\psi^* \circ \psi^* = \id
\end{equation*}
on $H^1_{\text{dR}}(S)$. Then, the results found in Section \ref{sec:sym-lin-alg} apply. In particular, for $\sigma \in \{+1,-1\}$, setting
\begin{equation*}
E_\sigma = \{C \in H^1_{\text{dR}}(S) : \phi^*C = \sigma C\}, 
\end{equation*}
we have that
\begin{equation}\label{eq:lattice-splitting-for-1st-cohomology}
\begin{split}
E_{+1} \oplus E_{-1} &= H^1_{\text{dR}}(S),\\
\myspan{(E_{\sigma} \cap H^1_{\text{dR},\ZZ}(S))} &= E_{\sigma},\\
\dim E_\sigma &= \frac{b}{2}.
\end{split}
\end{equation}

We now state and prove an extension of Proposition \ref{prop:periodic-orbits}.

\begin{proposition}
Let $\mu \in \Omega^2(S)$ be an area-form and $X$ a vector field on $S$ such that
\begin{equation*}
\cL_X \mu = 0.    
\end{equation*}
Suppose that $(\psi_* X)|_x$ and $X|_x$ are linearly-dependent at each point $x \in S$. Then, if $X$ is not identically zero, $X$ has a non-trivial periodic orbit.
\end{proposition}

\begin{proof}
In the current situation, we have that $b = \dim H^1_{\text{dR}}(S) = 2$. Hence, from Equation \eqref{eq:lattice-splitting-for-1st-cohomology}, we there exist closed 1-forms $\omega_1,\omega_2 \in \Omega^1(S)$ such that
\begin{enumerate}
    \item $[\omega_1]$ and $[\omega_2]$ have integral periods,
    \item $[\omega_1]$ and $[\omega_2]$ form a basis for $H^1_{\text{dR}}(S)$,
    \item $E_{+1} = \myspan\{[\omega_1]\}$, and
    \item $E_{-1} = \myspan\{[\omega_2]\}$.
\end{enumerate}
Hence, if $\omega$ is a closed 1-form which satisfies
\begin{equation*}
[\psi^*\omega] = \sigma[\omega],    
\end{equation*}
for some $\sigma \in \{+1,-1\}$, then $[a \omega]$ has integral periods for some $a \neq 0$.

Concerning the vector field $X$, we will now show that $\omega = i_X \mu$ is such that $[a\omega]$ has integral periods for some $a \in \RR \setminus \{0\}$. Indeed, first observe that because $\mu$ is an area-form, there exists $f \in C^{\infty}(S)$ such that
\begin{equation*}
\psi^*\mu = f\mu.  
\end{equation*}
With this, we compute
\begin{align*}
\psi^*\omega &= \psi^*(i_X \mu)\\
&= i_{(\psi^{-1})_* X}f\mu\\
&= f i_{(\psi^{-1})_* X}\mu.
\end{align*}
Now, because $\psi_* X$ and $X$ are linearly-dependent at each point of $S$, we have that $(\psi^{-1})_* X$ and $X$ are linearly-dependent at each point of $S$ and so $\omega = i_X\mu$ satisfies
\begin{equation}\label{eq:lin-dependence-of-closed-forms}
(\psi^*\omega) \wedge \omega = 0.    
\end{equation}
With respect to the symplectic form $W : H^1_{\text{dR}}(S) \times H^1_{\text{dR}}(S) \to \RR$, Equation \eqref{eq:lin-dependence-of-closed-forms} implies that
\begin{equation*}
W([\psi^*\omega],[\omega]) = 0.    
\end{equation*}
Because $H^1_{\text{dR}}(S)$ is 2-dimensional, this means that $[\psi^*\omega]$ and $[\omega]$ are linearly dependent cohomology classes. If $[\omega] = 0$, then $[\omega]$ trivially has integral periods. If $[\omega] \neq 0$, then linear dependence implies that
\begin{equation}\label{eq:lin-dependence-of-cohomologies}
[\psi^*\omega] = r [\omega],   
\end{equation}
for some $r \in \RR$. Because we are assuming that $\psi^* \circ \psi^* : H^1_{\text{dR}}(S) \to H^1_{\text{dR}}(S)$ is identity, applying $\psi^*$ to Equation \eqref{eq:lin-dependence-of-cohomologies} yields that
\begin{equation}
[\omega] = [\psi^*\psi^*\omega] = r^2[\omega].   
\end{equation}
Since we are assuming $[\omega] \neq 0$, this implies that $r^2 = 1$ and so $r \in \{+1,-1\}$. Hence, applying the discussion in the first paragraph, Equation \eqref{eq:lin-dependence-of-cohomologies} implies that $[a\omega]$ has integral periods for some $a \in \RR \setminus \{0\}$. 

In summary (in either case), $[a\omega]$ has integral periods for some $a \in \RR \setminus \{0\}$ and so by Lemma \ref{lem:periodic-orbits-for-abstract}, if $X$ is not identically zero, $X$ has a non-trivial periodic orbit.
\end{proof}

\appendix

\section{Proof of Lemma \ref{lem:homolgy-group-invariants} from Section \ref{sec:poloidal-and-toroidal-homologies}}\label{app:deferred-proofs-1}

For convenience, we will recall the statement of the lemma.
\homgroupinvar*
\begin{proof}
Because $\cI(S) = S'$, $\cI$ restricts to a diffeomorphism $\RR^n \setminus S \to \RR^n \setminus S'$ and thus preserves the set of connected components between $\RR^n \setminus S$ and $\RR^n \setminus S'$. That is,
\begin{equation*}
\{\cI(\cB_S),\cI(\cE_S)\} = \{\cB_{S'},\cE_{S'}\}.     
\end{equation*}
Therefore,
\begin{equation*}
\{\cI(\overline{\cB_S}),\cI(\overline{\cE_S})\} = \{\overline{\cB_{S'}},\overline{\cE_{S'}}\}.     
\end{equation*}
Now, $\overline{\cB_s}$ is the only compact element of $\{\overline{\cB_S},\overline{\cE_S}\}$. Therefore, Equation \eqref{eq:diffeo-preserves-int-and-ext} holds.

Before getting to orientation, we must make some preliminary observations. If $R \subset \RR^n$ is a compact regular domain in $\RR^n$, then we may give $R$ the induced orientation from $\RR^n$. The boundary $\partial R$ may then be given the induced orientation from $R$. Let $\iota_R : \partial R \subset \RR^n$ denote the inclusion. Then, for any $x \in \partial R$ and curve $\gamma : (-1,1) \to \RR^n$ such that, 
\begin{equation}\label{eq:admissible-curves}
\gamma((-1,0]) \subset R, \qquad \gamma(0) = x, \qquad \gamma'(0) \notin T_x\iota_R (T_x \partial R),
\end{equation}
a basis $v_1,...,v_{n-1}$ of $T_x \partial R$ has the induced orientation if and only if the basis
\begin{equation*}
T_x\iota_R (v_1),...,T_x\iota_R (v_{n-1}),\gamma'(0) 
\end{equation*}
of $T_x \RR^n$ has the standard orientation.

Because $\overline{\cB_S},\overline{\cB_{S'}} \subset \RR^n$ are compact regular domains, we may give $S = \partial \overline{\cB_S}$ and $S' = \partial \overline{\cB_{S'}}$ the induced orientations described in the previous paragraph. From Equation \eqref{eq:diffeo-preserves-int-and-ext}, $x \in S$ and $\gamma : (-1,1) \to S$ together satisfy Equation \eqref{eq:admissible-curves} with $R = \overline{\cB_S}$ if and only if $\tilde{x} = \cI(x) = \psi(x) \in S'$ and $\tilde{\gamma} = \cI \circ \gamma : (-1,1) \to \RR^n$ satisfy Equation \eqref{eq:admissible-curves} with $R = \overline{\cB_{S'}}$. From this, it is clear that $\cI$ is orientation-preserving if and only if $\psi$ is orientation-preserving.

To establish Equation \eqref{eq:diffeo-preserves-int-and-ext-homology}, we first note that, from Equation \eqref{eq:diffeo-preserves-int-and-ext}, we obtain the the diffeomorphisms
\begin{align*}
\cI_\cB &: \overline{\cB_S} \to \overline{\cB_{S'}},\\
\cI_\cE &: \overline{\cE_S} \to \overline{\cE_{S'}},
\end{align*}
induced from $\cI$. Consider the inclusions
\begin{align*}
\jmath_{\cB_S} &: S \subset \overline{\cB_S}, & \jmath_{\cE_S} &: S \subset \overline{\cE_S}\\
\jmath_{\cB_{S'}} &: S' \subset \overline{\cB_{S'}}, & \jmath_{\cE_{S'}} &: S' \subset \overline{\cE_{S'}}
\end{align*}
Then, we have the commutativity,
\begin{align*}
\cI_{\cB} \circ \jmath_{\cB_S} &= \jmath_{\cB_{S'}} \circ \psi,\\
\cI_{\cE} \circ \jmath_{\cE_S} &= \jmath_{\cE_{S'}} \circ \psi.
\end{align*}
Letting $k \in \NN_0$, we have on $H_k(S)$ that
\begin{align*}
{\cI_\cB}_* \circ {\jmath_{\cB_S}}_* &= {\jmath_{\cB_{S'}}}_* \circ \psi_*,\\
{\cI_\cE}_* \circ {\jmath_{\cE_S}}_* &= {\jmath_{\cE_{S'}}}_* \circ \psi_*
\end{align*}
In particular, because $\cI_\cB$, $\cI_\cE$ and $\psi$ are all diffeomorphisms, for all $\gamma \in H_k(S)$, one has that
\begin{align*}
{\jmath_{\cB_S}}_*(\gamma) = 0 &\iff {\jmath_{\cB_{S'}}}_*(\psi_*(\gamma)) = 0,\\
{\jmath_{\cE_S}}_*(\gamma) = 0 &\iff {\jmath_{\cE_{S'}}}_*(\psi_*(\gamma)) = 0.
\end{align*}
From this, Equation \eqref{eq:diffeo-preserves-int-and-ext-homology} follows.
\end{proof}

\section{Proof of Lemma \ref{lem:diffeomorphism-action-compatibility} from Section \ref{sec:extension-to-nested-toroidal-surfaces}}\label{app:deferred-proofs-2}

For this section, the definitions of poloidal and toroidal subgroups and indices are needed. They are found in the introduction in Proposition \ref{prop:homology-action-characterisation} and Definition \ref{def:pol-tor-index}. We now recall the statement of Lemma \ref{lem:diffeomorphism-action-compatibility}.
\diffactcom*
To prove the lemma, we first establish some sublemmas.

\begin{lemma}\label{lem:isotopy}
Let $n \in \mathbb{N}$ and $S$ be a closed and connected manifold of dimension $n-1$. Let $I$ be an open interval containing $0$, and $\Sigma : S \times I \to \RR^n$ be an embeddeding. Then, for each $i \in I$, there exists a smooth isotopy $H : [0,1] \times \RR^n \to \RR^n$ of $\RR^n$ such that, for all $x \in S$ and $s \in [0,1]$,
\begin{align*}
H(0,\cdot) &= \id,\\
H(1,\Sigma(x,0)) &= \Sigma(x,i),\\
H(\{s\} \times U) &= U,
\end{align*}
where $U = \Sigma(S\times I)$.
\end{lemma}

\begin{proof}
Fix $i \in I$ and write $I = (a,b)$ with $a < 0 < b$. Let $d = \min\{-a, i-a,b-i,b\} > 0$. We will construct the isotopy from a vector field. To this end, let $\rho : I \to \RR$ be a smooth function such that
\begin{align*}
\rho|_{(a+d/2,b-d/2)} &= 1,\\
\rho|_{(a,a+d/3)} &= 0,\\
\rho|_{(b-d/3,b)} &= 0.
\end{align*}
Define the function $P : S \times I \to \RR$ given by
\begin{equation*}
P(x,t) = \rho(x,t).    
\end{equation*}
Let $\partial_t$ denote the vector field on $S \times I$ inherited from the interval $I \subset \RR$. Then, consider the vector field
\begin{equation*}
X = P \partial_t    
\end{equation*}
on $S \times I$. Observe that $X$ is a compactly supported vector field on $S \times I$ and therefore has a global flow $\psi^X$ on $S \times I$. Because $0,i \in (a+d/2,b-d/2)$, for $x \in S$ we find that
\begin{equation*}
\psi^X_{i}(x,0) = (x,i).     
\end{equation*}
Because $X$ is compactly supported, we obtain a vector field $V$ in $\RR^n$ which is compactly supported in $U = \Sigma(S \times I)$ which satisfies
\begin{equation*}
V \circ \Sigma = T\Sigma \circ X.    
\end{equation*}
In particular, $V$ has a global flow $\psi^V : \RR \times \RR^n \to \RR^n$. An isotopy $H : [0,1] \times \RR^n \to \RR^n$ with the desired properties is therefore given by
\begin{equation*}
H(s,y) = \psi^V(is, y).
\end{equation*}
\end{proof}

As a general notation used in this section, if $S$ is a manifold, $J$ is an interval, $\Sigma : S \times J \to \RR^n$ is an embedding, and $\gamma : [0,1] \to S$ is a smooth curve, then for $j \in J$, we let $\gamma^j : [0,1] \to \Sigma(S \times \{j\})$ be the smooth curve given by
\begin{equation}
\gamma^j(t) = \Sigma(\gamma(t),j).    
\end{equation}
Although more general statements can be derived, we now apply Lemma \ref{lem:isotopy} to the 2-torus.

\begin{lemma}\label{lem:poloidal-toroidal-generators-open-int}
Let $S$ be a 2-torus. Let $I$ be an open interval containing $0$ and $\Sigma : S \times I \to \RR^3$ be an embedding. Let $S_i = \Sigma(\{i\} \times S)$ for $i \in I$. Denote the poloidal and toroidal subgroups of $S_i$ as $\cP_i, \cT_i \subset H_1(S_i)$ for $i \in I$. Let $\gamma_P,\gamma_T : [0,1] \to S$ be smooth closed curves such that the homology classes $[\gamma_P^0],[\gamma_T^0]$ respectively generate the poloidal and toroidal subgroups $\cP_0$ and $\cT_0$ of $S_0$. Then, for $i \in I$, $[\gamma_P^i]$ and $[\gamma_T^i]$ respectively generate the poloidal and toroidal subgroups $\cP_i$ and $\cT_i$ of $S_i$.
\end{lemma} 

\begin{proof}
Let $i \in I$. Then, consider an isotopy $H : [0,1] \times \RR^n \to \RR^n$ as in Lemma \ref{lem:isotopy}. Then, consider the diffeomorphism $H_1 = H(1,\cdot) : \RR^n \to \RR^n$. We have that $H_1(S_0) = S_i$. Letting $\psi : S_0 \to S_i$ denote the diffeomorphism induced from $H_1$, we have immediately from Equation \eqref{eq:diffeo-preserves-int-and-ext-homology} in Lemma \ref{lem:homolgy-group-invariants} that
\begin{equation*}
\psi_*(\cP_0) = \cP_i, \qquad \psi_*(\cE_0) = \cE_i.
\end{equation*}
On the other hand, observe that for any smooth curve $\gamma : [0,1] \to S$, one has that
\begin{equation*}
\gamma^i = \psi \circ \gamma^0.
\end{equation*}
Hence, because the homology classes $[\gamma_P^0],[\gamma_T^0]$ respectively generate the poloidal and toroidal subgroups $\cP_0$ and $\cT_0$ of $S_0$, the homology classes
\begin{equation}
[\gamma^i_P] = \psi_*[\gamma^0_P] \in \cP_i, \qquad [\gamma^i_T] = \psi_*[\gamma^0_T] \in \cT_i,
\end{equation}
must respectively generate $\cP_i$ and $\cT_i$.
\end{proof}

To prove Lemma \ref{lem:diffeomorphism-action-compatibility} we will also make use of a standard elementary result used in differential geometry.

\begin{lemma}\label{lem:injective-slice}
Let $S$ be a compact second countable Hausdorff topological space. Let $I \subset \RR$ be an open interval containing zero and let $M$ a Hausdorff topological space. Let $\Sigma : S \times I \to M$ be a locally-injective continuous map such that, the restricted map $\Sigma|_{S \times \{0\}}$ is injective. Then, there exists $\delta > 0$ with $(-\delta,\delta) \subset I$ such that $\Sigma|_{S \times (-\delta,\delta)}$ is injective.
\end{lemma}

\begin{proof}
Suppose not. Then, there exists sequences $\{(x_n,s_n)\}_{n \in \mathbb{N}}$ and $\{(y_n,t_n)\}_{n \in \mathbb{N}}$ in $S \times I$ such that, for all $n \in \mathbb{N}$,
\begin{align*}
\Sigma(x_n,s_n) &= \Sigma(y_n,t_n), & (x_n,s_n) &\neq (y_n,t_n), & |s_n| &< 1/n, & |t_n| &< 1/n.
\end{align*}
Then, because the topological space $(S \times [0,1])^2$ is compact, second countable, and Hausdorff, there exists a convergent subsequence $\{(x_{n_m},s_{n_m},y_{n_m},t_{n_m})\}_{m \in \mathbb{N}}$ of $\{(x_{n},s_{n},y_{n},t_{n})\}_{n \in \mathbb{N}}$ in $(S \times [0,1])^2$. Because $s_n \to 0$ and $t_n \to 0$ in $I$, it follows that
\begin{equation*}
(x_{n_m},s_{n_m},y_{n_m},t_{n_m}) \to (x,0,y,0),
\end{equation*}
in $(S \times I)^2$, for some $x,y \in S$. By continuity of $\Sigma$, we have in $M$ that
\begin{align*}
\Sigma(x_{n_m},s_{n_m}) &\to \Sigma(x,0),\\
\Sigma(y_{n_m},t_{n_m}) &\to \Sigma(y,0).
\end{align*}
On the other hand, because $\Sigma(x_n,s_n) = \Sigma(y_n,t_n)$ for $n \in \mathbb{N}$, we have (by uniqueness of limits in Hausdorff spaces) that $\Sigma(x,0) = \Sigma(y,0)$, and therefore, because $\Sigma|_{S \times \{0\}}$ is injective, $x = y$. Hence, in $S \times I$, we have
\begin{equation}\label{eq:conv}
\begin{split}
(x_{n_m},s_{n_m}) &\to (x,0),\\
(y_{n_m},t_{n_m}) &\to (x,0).
\end{split}
\end{equation}
Because $\Sigma$ is a local injection, there exists a neighborhood $U$ of $(x,0)$ in $S \times I$ such that $\Sigma|_{U}$ is injective. However, from Equation \eqref{eq:conv}, because $U$ is a neighborhood of $(x,0)$, there exists $m_0 \in \mathbb{N}$ such that, for all $m > m_0$, we have
\begin{equation*}
(x_{n_m},s_{n_m}), (y_{n_m},t_{n_m}) \in U.
\end{equation*}
However, for all $n \in \mathbb{N}$, we have $\Sigma(x_n,s_n) = \Sigma(y_n,t_n)$. So, by injectivity of $\Sigma|_U$, for $m > m_0$, we have that $(x_{n_m},s_{n_m}) = (y_{n_m},t_{n_m})$, which contradicts the fact that $(x_n,s_n) \neq (y_n,t_n)$ for all $n \in \mathbb{N}$. Hence, by contradiction, we are done.
\end{proof}

The previous lemma enables the following extension lemma for embeddings.

\begin{lemma}\label{lem:extension-lemma-for-embeddings}
Let $n \in \mathbb{N}$ and $S$ be a closed and connected manifold of dimension $n-1$. Let $\epsilon > 0$ and $J = [0,\epsilon)$. Let $\Sigma : S \times J \to \RR^n$ be an embedding. Then, there exists $\delta > 0$ and an embedding $\Sigma_2 : S \times (-\delta,\epsilon) \to \RR^n$ such that $\Sigma_2|_{S \times J} = \Sigma$.
\end{lemma}

\begin{proof}
We will use repeatedly the following observation: let $F : M \to N$ be a local diffeomorphism between manifolds which is also injective. Because $F$ is a local diffeomorphism, it is an open map. In particular, the image $V = F(M)$ is an open subset of $N$. The map $f : M \to V$ induced by $F$ is a local diffeomorphism which is bijective. Thus, $f$ is a diffeomorphism. In terms of $F$, this means that $F$ is a diffeomorphism onto its open image.

Now, because $J \times S$ is a closed subset of $\tilde{J} = (-\epsilon,\epsilon)$, there exists a smooth map $\tilde{\Sigma} : S \times \tilde{J} \to \RR^n$ extending $\Sigma$ (by the Extension Lemma for Smooth Functions) \cite[Lemma 2.26]{Lee12}).

Let $U$ denote the set of points $z \in S \times \tilde{J}$ such that $T_z \tilde{\Sigma} : T_z(S \times \tilde{J}) \to T_{\Sigma(z)}\RR^n$ is bijective. For $z \in S \times J$, we have that $T_z \Sigma : T_z(S \times J) \to T_{\Sigma(z)}\RR^n$ is bijective because $\Sigma$ is an embedding (between manifolds with boundary of the same dimension). Thus, the set $U$ contains all of $S \times J$. On the other hand, $U$ is open, and so, $U$ is a neighbourhood of $S \times \{0\}$ in $S \times \tilde{J}$. In particular, because $S$ is compact, there exists an open interval $I$ containing $0$ such that $S \times I \subset U$. However, by the Inverse Function Theorem, this means that $\tilde{\Sigma}|_{S \times I} : S \times I \to \RR^n$ is a local diffeomorphism, and is, in particular, locally injective. Because $\tilde{\Sigma}|_{S \times \{0\}} = \Sigma|_{S \times \{0\}}$ is injective, it follows from Lemma \ref{lem:injective-slice} that there exists $0 < \delta < \epsilon$ such that $\tilde{\Sigma}|_{S \times (-\delta,\delta)}$ is injective.

Now, consider the interval $I = (-\delta,\epsilon)$. We wish to show that $\Sigma_2 = \tilde{\Sigma}|_{S \times (-\delta,\epsilon)}$ is an embedding. To this end, first note that, because $\tilde{\Sigma}|_{S \times (-\delta,\delta)}$ is injective, it is a diffeomorphism onto its open image. Now, consider $S_0 = \Sigma(S \times \{0\})$ and let $\cB_{S_0}$ and $\cE_{S_0}$ be as in Theorem \ref{thm:detailed-separation}. Then, because $\overline{\cB_{S_0}}$ and $\overline{\cE_{S_0}}$ are manifolds with common boundary 
$S_0$, we cannot have that $\tilde{\Sigma}(S \times (-\delta,\delta)) \subset \overline{\cB_{S_0}}$ or $\tilde{\Sigma}(S \times (-\delta,\delta)) \subset \overline{\cE_{S_0}}$. On the other hand, $\tilde{\Sigma}(S \times (-\delta,0))$ and $\tilde{\Sigma}(S \times (0,\delta))$ are disjoint connected subsets of $\mathbb{R}^n \setminus S_0$. Thus, we must either have that
\begin{equation*}
\tilde{\Sigma}(S \times (-\delta,0)) \subset \cB_{S_0}, \qquad \tilde{\Sigma}(S \times (0,\delta)) \subset \cE_{S_0},    
\end{equation*}
or
\begin{equation*}
\tilde{\Sigma}(S \times (-\delta,0)) \subset \cE_{S_0}, \qquad \tilde{\Sigma}(S \times (0,\delta)) \subset \cB_{S_0}.
\end{equation*}
In either case, because $\tilde{\Sigma}(S \times (0,\epsilon))$ contains $\tilde{\Sigma}(S \times (0,\delta))$, we therefore have that
\begin{equation*}
\tilde{\Sigma}(S \times (-\delta,0)) \cap \tilde{\Sigma}(S \times (0,\epsilon)) = \emptyset.
\end{equation*}
Hence, because $\tilde{\Sigma}|_{S \times (-\delta,\delta)}$ is injective,
\begin{equation*}
\tilde{\Sigma}(S \times (-\delta,0)) \cap \tilde{\Sigma}(S \times [0,\epsilon)) = \emptyset.
\end{equation*}
However, $\tilde{\Sigma}|_{S \times (-\delta,0)}$ and $\tilde{\Sigma}|_{S \times [0,\epsilon)} = \Sigma$ are both injective. Therefore, $\tilde{\Sigma}|_{S \times (-\delta,\epsilon)}$ is injective. However, this means that $\tilde{\Sigma}|_{S \times (-\delta,\epsilon)}$ is a diffeomorphism onto its open image. That is, $\Sigma_2 = \tilde{\Sigma}|_{S \times (-\delta,\epsilon)}$ is an embedding, as claimed. It is also clear that $\Sigma_2$ extends $\Sigma$.
\end{proof}

We immediately obtain the following as a corollary of \ref{lem:extension-lemma-for-embeddings}.

\begin{lemma}\label{lem:poloidal-toroidal-generators-half-open-int}
Lemma \ref{lem:poloidal-toroidal-generators-open-int} still holds when replacing ``open interval containing $0$" with ``interval of the form $[0,\epsilon)$ for some $\epsilon > 0$".
\end{lemma} 

We are now ready to prove Lemma \ref{lem:diffeomorphism-action-compatibility}.

\begin{proof}[Proof of Lemma \ref{lem:diffeomorphism-action-compatibility}]
For $j \in J$, let $\cP_j$ and $\cT_j$ denote the poloidal and toroidal subgroups of $S_j$. Let $\gamma_P,\gamma_T : [0,1] \to S$ be smooth closed curves such that the homology classes $[\gamma_P^0]$ and $[\gamma_T^0]$ respectively generate $\cP_0$ and $\cT_0$. By lemmas \ref{lem:poloidal-toroidal-generators-open-int} and \ref{lem:poloidal-toroidal-generators-half-open-int}, $[\gamma_P^j]$ generates $\cP_j$ and $[\gamma_T^j]$ generates $\cT_j$. Hence, if $\psi_j : S_j \to S_j$ denotes the diffeomorphism induced by $\cI$, then
\begin{align*}
{\psi_j}_* [\gamma_P^j] &= (\sigma_1)_j[\gamma_P^j],\\
{\psi_j}_* [\gamma_T^j] &= (\sigma_2)_j[\gamma_T^j].
\end{align*}
Let $U = \Sigma(S \times J)$ and $\Psi : U \to U$ denote the diffeomorphism induced from $\cI$ and $\iota_j : S_j \subset U$ is the inclusion. Then
\begin{equation}\label{eq:action-of-Psi-for-all-j}
\begin{split}
\Psi_*[\iota_j \circ \gamma_P^j] =  {\iota_j}_*{\psi_j}_*[\gamma_P^j] &= (\sigma_1)_j~{\iota_j}_*[\gamma_P^j] = (\sigma_1)_j [\iota_j \circ \gamma_P^j],\\
\Psi_*[\iota_j \circ \gamma_T^j] = {\iota_j}_*{\psi_j}_*[\gamma_T^j] &= (\sigma_2)_j~{\iota_j}_*[\gamma_T^j] = (\sigma_2)_j[\iota_j \circ \gamma_T^j].
\end{split}    
\end{equation}

On the other hand, for $j \in J$, through the map $\Sigma$, $\iota_0 \circ \gamma_P^0$ is homologous to $\iota_j \circ \gamma_P^j$ in $U$ and $\iota_0 \circ \gamma_T^0$ is homologous to $\iota_j \circ \gamma_T^j$ in $U$. Hence, Equation \eqref{eq:action-of-Psi-for-all-j} implies that
\begin{equation}\label{eq:conclusion-from-Psi-action}
\begin{split}
(\sigma_1)_j [\iota_j \circ \gamma_P^j] &= \sigma_1 [\iota_0 \circ \gamma_P^0],\\
(\sigma_2)_j [\iota_j \circ \gamma_T^j] &= \sigma_2 [\iota_0 \circ \gamma_T^0].
\end{split}    
\end{equation}
Moreover, for $j \in J$, the map $\Sigma$ gives a smooth deformation retract $U \to S_j$ for $j \in J$. In particular, the inclusions on first homology ${\iota_j}_* : H_1(S_j) \to H_1(U)$ are isomorphisms. Hence, each homology class in Equation \eqref{eq:conclusion-from-Psi-action} is non-zero. Thus, for all $j \in J$,
\begin{equation*}
(\sigma_1)_j = \sigma_1, \qquad (\sigma_2)_j = \sigma_2.     
\end{equation*}
\end{proof}

\section{Solid tori which only posses inversion symmetry}\label{app:for-the-figure}

The purpose of this section is to provide an exposition of basic techniques used to prove the claim about the solid torus in Figure \ref{fig:solid-torus-example} of the introduction. Here we deal with isometries in Euclidean space where the metric is taken to be the standard Euclidean one. For example, the following result is well-known.

\begin{lemma}\label{lem:char-of-isom}
Let $\cI : \RR^n \to \RR^n$ be a map. Then, $\cI$ is an isometry if and only if there exists a matrix $A \in O(n)$ and vector $c \in \RR^n$ such that, for $x \in \RR^n$,
\begin{equation*}
\cI(x) = Ax + c.    
\end{equation*}
\end{lemma}

We will consider the following standard type of tubular neighbourhood.

\begin{proposition}\label{prop:tubular-nbhd}
Let $\gamma \subset \RR^3$ be an embedded circle and $R > 0$. Consider the set
\begin{equation*}
N(\gamma,R) = \{x \in \RR^3 : \dist(x,\gamma) \leq R\}.
\end{equation*}
Let $t : \gamma \to \RR^3$ be a smooth unit tangent vector to $\gamma$. For $x \in \gamma$, set
\begin{equation*}
D_{R,x} = \{y \in \RR^3 : (x-y) \cdot t(x) = 0 \text{ and } \|x-y\| \leq R\}.    
\end{equation*}
Then, for $R$ small-enough,
\begin{equation*}
N(\gamma,R) = \cup_{x \in \gamma}D_{R,x}
\end{equation*}
and $N(\gamma,R)$ is an embedded solid torus with boundary
\begin{equation*}
\partial N(\gamma,R) = \{x \in \RR^3 : \dist(x,\gamma) = R\}.    
\end{equation*}
\end{proposition}

One can reconstruct the curve $\gamma$ from $N(\gamma,R)$ in Proposition \ref{prop:tubular-nbhd} in the following standard way.

\begin{lemma}\label{lem:reconstrucing-the-curve}
Let $N(\gamma,R)$ be a tubular neighbourhood with $R$ small-enough as in Proposition \ref{prop:tubular-nbhd}. Let $n : \partial N(\gamma,R) \to \RR^3$ be the outward unit normal field of $\partial N(\gamma,R)$ in $\RR^3$. Then,
\begin{equation*}
\gamma = \{y- R~n(y) : y \in \partial N(\gamma,R)\}.     
\end{equation*}
\end{lemma}

Concerning isometries, Lemma \ref{lem:reconstrucing-the-curve} implies that isometries of tubular neighbourhoods are the isometries of the inner curve. More precisely, we have the following.

\begin{corollary}\label{cor:isometry-reduction}
Let $R > 0$ and $\gamma_1,\gamma_2 \subset \RR^3$ be embedded circles. Then, for $R$ small-enough,
\begin{equation*}
N(\gamma_1,R) = N(\gamma_2,R) \iff \gamma_1 = \gamma_2.
\end{equation*}
In particular, for any isometry $\cI : \RR^3 \to \RR^3$, for $R$ small-enough, one has that
\begin{equation*}
\cI(\gamma_1) = \gamma_1 \iff \cI(N(\gamma_1,R)) = N(\gamma_1,R). 
\end{equation*}
\end{corollary}

We now consider some characteristics of curves which restrict their Euclidean isometries. For a subset $K \subset \RR^n$ we set
\begin{equation*}
-K = \{-x : x \in K\}.    
\end{equation*}

\begin{lemma}\label{lem:isometries-map-to-zero}
Let $K \subset \RR^n$ be a compact subset. Let
\begin{align*}
\rho(K) = \max_{x \in K} \|x\|.
\end{align*}
Suppose that $-K = K$. Then, for any $x,y \in K$ with $d(x,y) = \diam K$, we have that
\begin{equation*}
x+y = 0, \quad d(x,y) = 2\rho(K).    
\end{equation*}
In particular, if $\cI : \RR^n \to \RR^n$ is an isometry such that $\cI(K) = K$, then
\begin{equation*}
\cI(0) = 0.    
\end{equation*}
In particular, $\cI$ is a linear transformation.
\end{lemma}

\begin{proof}
For $x,y \in \RR^n$,
\begin{equation*}
\|x-y\|^2 + \|x+y\|^2 = 2\|x\|^2+2\|y\|^2.    
\end{equation*}
Hence, for $x,y \in K$,
\begin{equation}\label{eq:rho-inequality}
\|x-y\| \leq 2\rho(K)
\end{equation}
with equality if and only if
\begin{equation}\label{eq:equalityiff}
x+y = 0. 
\end{equation}
Moreover, because $-K = K$, setting $x \in K$ with $\|x\| = \rho(K)$ and $y = -x \in K$, we have $d(x,y) = 2\rho(K)$. Therefore, by equations \eqref{eq:rho-inequality} and \eqref{eq:equalityiff}, $\diam K = 2\rho(K)$ and for any $x,y \in K$ with $d(x,y) = \diam K$, we have that
\begin{equation*}
x+y = 0, \quad d(x,y) = 2\rho(K),    
\end{equation*}
as claimed.

Concerning isometries, let $\cI : \RR^n \to \RR^n$ be an isometry such that $\cI(K) = K$. Fix $x \in K$ with $\|x\| = \rho(K)$. Then, setting $y=-x \in K$, we have that $d(x,y) = \diam K$. Therefore, because $\cI$ is an isometry,
\begin{equation*}
d(\cI(x),\cI(y)) = \diam K.    
\end{equation*}
Thus, by the previous paragraph,
\begin{equation*}
\cI(x) + \cI(y) = 0.    
\end{equation*}
On the other hand, from Lemma \ref{lem:char-of-isom}, $\cI$ is an affine transformation and so
\begin{equation*}
\cI\left(\frac{x+y}{2}\right) = \frac{\cI(x)+\cI(y)}{2}.
\end{equation*}
Therefore,
\begin{equation*}
\cI(0) = 0.
\end{equation*}
\end{proof}

We now consider a curve of the form
\begin{equation*}
\gamma(t) = (r(t) \cos t \sin \theta(t),r(t) \sin t  \sin\theta(t),r(t)\cos\theta(t)),
\end{equation*}
where $r,\theta : \RR \to \RR$ are smooth $2\pi$-periodic functions and $r > 0$. We give conditions on $r$ and $\theta$ such that the image
\begin{equation*}
\Gamma = \gamma(\RR)    
\end{equation*}
is an embedded circle in $\RR^3$ and does not possess any Euclidean isometry except the identity and the isometry $(x,y,z) \mapsto (-x,-y,-z)$. We will use repeatedly that
\begin{equation*}
\|\gamma(t)\| = r(t).    
\end{equation*}
We will consider for $i,j,k \in \{+1,-1\}$, the isometry given by
\begin{equation*}
\cI_{i,j,k}(x,y,z) = (ix,jy,kz).    
\end{equation*}

\begin{lemma}\label{lem:regularity-of-curve}
Assume that $0 < \theta(t) < \pi$ for all $t \in \RR$. Then, curve $\gamma$ is a regular closed curve with period $2\pi$, in the sense that
\begin{enumerate}
    \item $\gamma'(t) \neq 0$ for all $t \in \RR$,
    \item $\gamma(2\pi + t) = \gamma(t)$ for all $t \in \RR$,
    \item $\gamma|_{[0,2\pi)}$ is injective. \label{item:injective}
\end{enumerate}
In particular, the image of $\gamma$ is a smoothly embedded circle in $\RR^3$. In fact, a stronger property holds than \ref{item:injective}. Namely, for any $i,j,k \in \{+1,-1\}$ and for any $s,t \in \RR$ such that
\begin{equation*}
\gamma(s) = \cI_{i,j,k}(\gamma(t)),    
\end{equation*}
one has that
\begin{equation*}
r(s) = r(t), \quad \cos s = i\cos t, \quad \sin s = j\sin t, \quad \cos \theta(s) = k\cos\theta(t).
\end{equation*}
So that, if $i = 1 = j$, then $s-t \in 2\pi\ZZ$.
\end{lemma}

\begin{proof}
One computes that
\begin{equation*}
\|\gamma'(t)\|^2 = r'(t)^2 + r(t)^2(\sin(\theta(t))^2 + \theta'(t)^2).    
\end{equation*}
Hence, for all $t \in \RR$, because $0<\theta(t)<\pi$, we have that $\sin(\theta(t)) > 0$ and so
\begin{equation*}
\|\gamma'(t)\|^2 \geq  r(t)^2 \sin(\theta(t))^2 > 0, 
\end{equation*}
because $r(t) > 0$. Because $r$ and $\theta$ are both $2\pi$-periodic, we see that $\gamma(t+2\pi) = \gamma(t)$ for $t \in \RR$. We now prove the claim stronger than property \ref{item:injective}, thereby also proving \ref{item:injective}. Let $i,j,k \in \{+1,-1\}$ and $s,t \in \RR$ satisfy
\begin{equation*}
\gamma(s) = \cI_{i,j,k}(\gamma(t))
\end{equation*}
that is,
\begin{equation}\label{eq:components-form}
\begin{split}
r(s) \cos s \sin \theta(s) &= ir(t) \cos t \sin \theta(t),\\
r(s) \sin s \sin\theta(s) &= jr(t) \sin t  \sin\theta(t),\\
r(s)\cos\theta(s) &= kr(t)\cos\theta(t).
\end{split}
\end{equation}
Squaring and adding all three equations in Equation \eqref{eq:components-form} yields
\begin{equation*}
r(s) = r(t)
\end{equation*}
because $r(s),r(t) > 0$. Squaring and adding the first two equations in Equation \eqref{eq:components-form} therefore yields $\sin^2\theta(s)=\sin^2\theta(t)$, and because $0 < \theta(s),\theta(t) < \pi$, we have that $\sin\theta(s),\sin\theta(t) > 0$, and therefore
\begin{equation*}
\sin \theta(s) = \sin \theta(t).
\end{equation*}
Inserting this information back into Equation \eqref{eq:components-form} yields
\begin{equation*}
\cos s = i\cos t, \quad \sin s = j\sin t, \quad \cos \theta(s) = k\cos\theta(t).     
\end{equation*}
\end{proof}

We now refine our choices of the functions $r$ and $\theta$ to limit the number of possible isometries to ones of the type $\cI_{i,j,k}$. 

\begin{lemma}\label{lem:reduce-to-ijk}
Assume the property of $\theta$ in assumed Lemma \ref{lem:regularity-of-curve}. Assume in addition that, for all $t \in \RR$,
\begin{equation}
r(t+\pi) = r(t), \quad \theta(t+\pi) = \theta(t).
\end{equation}
Assume also that the maximum of $r|_{[0,2\pi)}$ (which exists by continuity and periodicity) is only attained on $\{0,\pi\}$ while the minimum is only attained on $\{\pi/2,3\pi/2\}$. Lastly, assume that $\theta(t) = \pi/2$ for all the extremal points $t \in \{0,\pi/2,\pi,3\pi/2\}$ of $r|_{[0,2\pi)}$. Then, setting
\begin{equation*}
\Gamma = \gamma(\RR),
\end{equation*}
one has that
\begin{equation*}
-\Gamma = \Gamma,
\end{equation*}
and if $\cI : \RR^3 \to \RR^3$ is an isometry such that $\cI(\Gamma) = \Gamma$, then
\begin{equation*}
\cI = \cI_{i,j,k}    
\end{equation*}
for some $i,j,k \in \{+1,-1\}$.
\end{lemma}

\begin{proof}
We recall that for $t \in \RR$,
\begin{equation*}
\gamma(t) = (r(t) \cos t \sin \theta(t),r(t) \sin t  \sin\theta(t),r(t)\cos\theta(t)).
\end{equation*}
In the case at hand, for all $t \in \mathbb{R}$, $r(t+\pi) = r(t)$ and $\theta(t+\pi) = \theta(t)$. Therefore,
\begin{equation*}
\gamma(t+\pi) = -\gamma(t).
\end{equation*}
Hence, we have that
\begin{equation*}
-\Gamma = \Gamma.    
\end{equation*}

$\Gamma$ is compact and so we may set
\begin{equation*}
\rho(\Gamma) = \max_{x \in \Gamma} \|x\|, \qquad  m(\Gamma) = \min_{x \in \Gamma}\|x\|,
\end{equation*}
and
\begin{equation*}
\Gamma_\rho = \{x \in \Gamma : \|x\| = \rho(\Gamma)\}, \qquad \Gamma_{m} = \{x \in \Gamma : \|x\| = m(\Gamma)\}.
\end{equation*}
From our assumptions on $r$ and $\theta$,
\begin{align*}
\Gamma_\rho &= \{\gamma(0),\gamma(\pi)\} = \{r(0)e_1,-r(0)e_1\},\\
\Gamma_m &= \{\gamma(\pi/2),\gamma(3\pi/2)\} = \{r(\pi/2)e_2,-r(\pi/2)e_2\},
\end{align*}
where $e_1,e_2,e_3$ denote the standard basis of $\RR^3$. In particular, if $\cI : \RR^3 \to \RR^3$ is an isometry such that $\cI(\Gamma) = \Gamma$, then by Lemma \ref{lem:isometries-map-to-zero}, we have that $\cI(0) = 0$, and therefore, 
\begin{equation*}
\cI(\Gamma_m) = \Gamma_m, \qquad \cI(\Gamma_\rho) = \Gamma_\rho.
\end{equation*}
It follows that,
\begin{equation*}
\cI(\{e_1,-e_1\}) = \{e_1,-e_1\}, \qquad \cI(\{e_2,-e_2\}) = \{e_2,-e_2\},
\end{equation*}
and therefore also that
\begin{equation*}
\cI(\{e_3,-e_3\}) = \{e_3,-e_3\}.
\end{equation*}
Hence, for some $i,j,k \in \{+1,-1\}$,
\begin{equation*}
\cI = \cI_{i,j,k}.    
\end{equation*}
\end{proof}

We now give one last refinement to $r$ and $\theta$ such that the only isometries of $\Gamma = \gamma(\RR)$ are the identity and $\cI_{-1,-1,-1}$.

\begin{proposition}\label{prop:reduce-to-1}
Assume the properties of $r,\theta$ in assumed Lemma \ref{lem:reduce-to-ijk}. In addition, assume that
\begin{equation*}
\theta(\pi/4) \neq \theta(3\pi/4), \quad \theta(\pi/4) \neq \theta(7\pi/4). 
\end{equation*}
Then, if $\cI : \RR^3 \to \RR^3$ is an isometry such that $\cI(\Gamma) = \Gamma$, either $\cI = \cI_{-1,-1,-1}$ or $\cI = \id$.
\end{proposition}

\begin{proof}
First note that, from Lemma \ref{lem:reduce-to-ijk}, for any isometry $\cI : \RR^3 \to \RR^3$ satisfying $\cI(\Gamma) = \Gamma$ is of the form
\begin{equation*}
\cI = \cI_{i,j,k},    
\end{equation*}
for some $i,j,k \in \{+1,-1\}$. From the same lemma,
\begin{equation}\label{eq:-1-sym}
\cI_{-1,-1,-1}(\Gamma) = \Gamma.    
\end{equation}
Now, for any $i,j,k \in \{+1,-1\}$, because
\begin{equation*}
\cI_{i,j,k} \circ \cI_{-1,-1,-1} = \cI_{-i,-j,-k},    
\end{equation*}
Equation \eqref{eq:-1-sym} implies that
\begin{equation*}
\cI_{i,j,k}(\Gamma) = \Gamma \iff \cI_{-i,-j,-k}(\Gamma) = \Gamma.
\end{equation*}
Hence, it suffices to prove that
\begin{align}
\cI_{-1,1,1}(\Gamma) &\neq \Gamma \label{eq:-1,1,1},\\
\cI_{1,-1,1}(\Gamma) &\neq \Gamma \label{eq:1,-1,1},\\
\cI_{1,1,-1}(\Gamma) &\neq \Gamma \label{eq:1,1,-1}.
\end{align}

To prove Equation \eqref{eq:-1,1,1}, consider the point $\gamma(\pi/4) \in \Gamma$ and suppose that
\begin{equation*}
\cI_{-1,1,1}(\gamma(\pi/4)) \in \Gamma.
\end{equation*}
Then, $\gamma(s) = \cI_{-1,1,1}(\gamma(\pi/4))$ for some $s \in [0,2\pi)$. By Lemma \ref{lem:regularity-of-curve}, we have that
\begin{equation*}
r(s) = r(\pi/4), \quad \cos s = -\cos {\pi/4}, \quad \sin s = \sin {\pi/4}, \quad \cos \theta(s) = \cos\theta(\pi/4).     
\end{equation*}
From the middle two equations, it follows that $s = 3\pi/4$. However, from our assumptions, $\theta(\pi/4) \neq \theta(3\pi/4)$, and because $0 < \theta(\pi/4),\theta(3\pi/4)<\pi$, we therefore have that
\begin{equation*}
\cos\theta(\pi/4) \neq \cos\theta(3\pi/4).
\end{equation*}
This is a contradiction with $\cos \theta(s) = \cos\theta(\pi/4)$. Hence, $\cI_{-1,1,1}(\gamma(\pi/4)) \notin \Gamma$. Hence, Equation \eqref{eq:-1,1,1} is proven.

To prove Equation \eqref{eq:1,-1,1}, consider again the point $\gamma(\pi/4) \in \Gamma$ and suppose that $\cI_{1,-1,1}(\gamma(\pi/4)) \in \Gamma$. Then, $\gamma(s) = \cI_{1,-1,1}(\gamma(\pi/4))$ for some $s \in [0,2\pi)$. By Lemma \ref{lem:regularity-of-curve}, we have that
\begin{equation*}
r(s) = r(\pi/4), \quad \cos s = \cos {\pi/4}, \quad \sin s = -\sin {\pi/4}, \quad \cos \theta(s) = \cos\theta(\pi/4).     
\end{equation*}
From the middle two equations, it follows that $s = 7\pi/4$. However, from our assumptions, $\theta(\pi/4) \neq \theta(7\pi/4)$ and because $0 < \theta(\pi/4),\theta(7\pi/4)<\pi$, we therefore have that
\begin{equation*}
\cos\theta(\pi/4) \neq \cos\theta(7\pi/4).
\end{equation*}
This is a contradiction with $\cos \theta(s) = \cos\theta(\pi/4)$. Hence, $\cI_{1,-1,1}(\gamma(\pi/4)) \notin \Gamma$. Hence, Equation \eqref{eq:1,-1,1} is proven.

To prove Equation \eqref{eq:1,1,-1}, observe that because $\theta(\pi/4) \neq \theta(3\pi/4)$, there exists $t \in [0,2\pi)$ such that $\theta(t) \neq \pi/2$. Because $0 < \theta(t) < \pi$, we therefore have that
\begin{equation}\label{eq:cos-not-zero}
\cos\theta(t) \neq 0.    
\end{equation}
Now, we show that $\cI_{1,1,-1}(\gamma(t)) \notin \Gamma$. Indeed, if $\cI_{1,1,-1}(\gamma(t)) \in \Gamma$, then, for some $s \in [0,2\pi)$, $\gamma(s) = \cI_{1,1,-1}(\gamma(t))$. By Lemma \ref{lem:regularity-of-curve}, we have that
\begin{equation*}
r(s) = r(t), \quad \cos s = \cos t, \quad \sin s = \sin t, \quad \cos \theta(s) = -\cos\theta(t).     
\end{equation*}
From the middle two equations, it follows that $s = t$. However, this is a contradiction with $\cos \theta(s) = -\cos\theta(t)$ by Equation \eqref{eq:cos-not-zero}. Hence, Equation \eqref{eq:1,1,-1} is proven, and we are done.
\end{proof}

We now consider the curve $\gamma$ given in the introduction in Figure \ref{fig:solid-torus-example}. That is, choose $r$ and $\theta$ to be given by
\begin{equation*}
r(t) = 2+\cos(2t), \quad \theta(t) = \frac{\pi}{2} + \frac{1}{5}\sin(2t).    
\end{equation*}
We see that $r,\theta$ are smooth and $2\pi$-periodic, $r>0$,
\begin{equation*}
\theta(\pi/4) = \frac{\pi}{2}+\frac{1}{5} \neq \frac{\pi}{2}-\frac{1}{5} = \theta(3\pi/4) = \theta(7\pi/4),    
\end{equation*}
and for all $t \in \RR$,
\begin{align*}
\theta(t) &\in (0,\pi),\\
r(t+\pi) &= r(t),\\
\theta(t+\pi) &= \theta(t),
\end{align*}
Moreover, the maximum of $r|_{[0,2\pi)}$ is only attained on $\{0,\pi\}$ while the minimum is only attained on $\{\pi/2,3\pi/2\}$ and
\begin{equation*}
\theta(t) = \frac{\pi}{2} \text{ for all } t \in \{0,\pi/2,\pi,3\pi/2\}.   
\end{equation*}
That is, the assumptions placed on $r$ and $\theta$ in Proposition \ref{prop:reduce-to-1} hold. Hence, the image $\Gamma = \gamma(\RR)$ has $\cI = \cI_{-1,-1,-1}$ as its only non-trivial Euclidean symmetry. Therefore, by Corollary \ref{cor:isometry-reduction}, for $R > 0$ small-enough, the tubular neighbourhood $N(\Gamma,R)$ is a solid torus which has $\cI = \cI_{-1,-1,-1}$ as its only non-trivial Euclidean symmetry, as asserted in Figure \ref{fig:solid-torus-example} in the introduction.

\bibliographystyle{plain}
\bibliography{Paper.bib}

\end{document}